\documentclass[a4paper,12pt]{amsart}
\usepackage{amsrefs}

\usepackage[none]{hyphenat}
\usepackage{a4wide,hyperref,amsmath,amssymb,enumitem, stmaryrd}
\usepackage{tikz}
\usetikzlibrary{decorations.pathreplacing,calligraphy,patterns}
\usetikzlibrary{shapes}
\tikzset{
			inner sep=1pt,semithick,
			vertex/.style={circle,draw,fill,minimum size= 1pt},
			vertexb/.style={rectangle,draw,fill, minimum size = 3pt},
			vertexd/.style={rectangle,draw,fill=white, minimum size = 3pt},			
			vertexc/.style={circle,draw,fill=white,minimum size= 1pt},			
			thickedge/.style={line width=0.73pt},
			thickeredge/.style={line width=1.5pt},					
			font=\tiny
}

\newcommand{\T}{\mathcal{T}}
\newcommand{\K}{\mathcal{K}}
\newcommand{\D}{\mathcal{D}}

\newcommand{\XX}{\mathbb{X}}
\newcommand{\LL}{\mathbb{L}}
\newcommand{\II}{\mathbb{I}}

\newif\ifdetails
\detailstrue
\newcommand{\DETAIL}[1]%
{\ifdetails\par\fbox{\begin{minipage}{0.9\linewidth}\textit{Detail:}
      #1\end{minipage}}\par\fi}
\newcommand{\TODO}[1]%
{\ifdetails\par\fbox{\begin{minipage}{0.9\linewidth}\textbf{TODO:}
      #1\end{minipage}}\par\fi}

\newtheorem{lemma}{Lemma}

\newtheorem{theorem}[lemma]{Theorem}

\newtheorem{definition}{Definition}

\newcommand{\Crt}{\operatorname{crt}}
\newcommand{\Cr}{\operatorname{cr}}
\newcommand{\lcr}{\overline{\operatorname{cr}}}

\newcommand{\old}[1]{{}}
\makeatletter
\DeclareRobustCommand{\cev}[1]{%
  {\mathpalette\do@cev{#1}}%
}
\newcommand{\do@cev}[2]{%
  \vbox{\offinterlineskip
    \sbox\z@{$\m@th#1 x$}%
    \ialign{##\cr
      \hidewidth\reflectbox{$\m@th#1\vec{}\mkern4mu$}\hidewidth\cr
      \noalign{\kern-\ht\z@}
      $\m@th#1#2$\cr
    }%
  }%
}
\makeatother

\title[Tanglegrams with a unique $1$-crossing critical subtanglegram]{Tanglegrams with a unique $1$-crossing-critical subtanglegram have tangle crossing number $1$}

\author[\'E. Czabarka, A. Helm, L. A. Sz\'ekely]{\'Eva Czabarka, Alec Helm, L\'aszl\'o A. Sz\'ekely}

\address{\'Eva Czabarka\\ Department of Mathematics \\ University of South Carolina \\ Columbia, SC 29208 \\ USA}
\email{czabarka@math.sc.edu}

\address{Alec Helm\\ Department of Mathematics \\ University of South Carolina \\ Columbia, SC 29208 \\ USA}
\email{ah191@math.sc.edu}

\address{L\'aszl\'o A. Sz\'ekely\\ Department of Mathematics \\ University of South Carolina \\ Columbia, SC 29208 \\ USA}
\email{szekely@math.sc.edu}


\begin{document}
\maketitle

\begin{abstract}{A tanglegram of size $n$ is a graph formed from two rooted binary trees with $n$ leaves each and a perfect matching between their leaf sets. 
Tanglegrams are used to model co-evolution in various settings.
A tanglegram layout is a straight line drawing where the two trees are drawn as plane trees with their leaf-sets on two parallel lines, and only the edges
of the matching may cross. 
The tangle crossing number of a tanglegram is the minimum crossing number among its layouts. 
It is known that tanglegrams have crossing number at least one precisely when they contain one of two size $4$ subtanglegrams, which we refer to as 
cross-inducing subtanglegrams. 
We show here that a tanglegram with exactly one cross inducing subtanglegram must have tangle crossing number exactly one,
and ask the question whether the tangle-crossing number of tanglegrams with exactly $k$ cross-inducing subtanglegrams is bounded  for every $k$. }
\end{abstract}

\section{Introduction}

The {\it crossing number} $\Cr({\mathcal D})$ of a graph drawing $\mathcal D$ 
is the sum of the number of crossings between unordered edge pairs, and the {\it crossing number} $\Cr(G)$ of a graph $G$
is the minimum crossing number over all of its drawings. 
A {\it subdivision} of a graph is obtained by replacing some of its edges with paths; the subdividing vertices are the internal vertices of these paths. It is obvious and well-known that the crossing number of a graph is the same as the crossing number of any of its subdivisions, and the  celebrated Kuratowski Theorem asserts that a graph has a positive crossing number if and only if it contains a subgraph that is isomorphic to a subdivision of $K_5$ (the complete graph on $5$ vertices) or of $K_{3,3}$ (the complete bipartite graph with $3$ vertices in each partition class). We will refer to a subgraph of $G$ that is isomorphic to a $K_{3,3}$ or a $K_5$ as a {\it cross-inducing subgraph}. The {\it rectilinear crossing number} $\lcr(G)$ of a graph $\mathcal G$ 
is the minimum crossing number over all of its straight line drawings, where the edges are drawn in as straight lines. F\'ary's theorem \cite{fary} states that
planar graphs have planar straight line drawings, and consequently the cross-inducing subgraphs with respect to the regular and rectilinear crossing number are the same. 

A graph $G$ is {\it $k$-crossing-critical} if $G$ has no degree $2$ vertices, $\Cr(G)\geq k$, but any proper subgraph has crossing number below $k$.
A graph has crossing number at least $k$ precisely when it contains a subdivision of a $k$-crossing critical graph.
The Kuratowski Theorem gives that
$1$-crossing-critical graphs are exactly $K_5$ and $K_{3,3}$. For $k\ge 2$ such finite characerizations do no exist:
Kochol \cite{kochol} gave an infinite family of $3$-connected $k$-crossing-critical graphs for any $k\ge 2$.
The 
$2$-crossing-critical graphs, up to finitely many exceptions, were characterized by Bokal, Oporowski, Richter and Salazar \cite{2cc}. 

A {\it tanglegram} $\T$ is a graph that consists of a left rooted binary tree $L_{\T}$ and a right rooted binary tree $R_{\T}$ with the same number of leaves, and a perfect matching $\sigma_{\T}$ between their leaves 
(for the definitions regarding trees, see Section~\ref{sec:treedefs},  regarding tanglegrams, see Section~\ref{sec:tangdefs}). The {\it size} of the tanglegram is 
$|\sigma_{\T}|$.
A
 {\it tanglegram layout} is a straight line drawing where the rooted left- and right-trees are drawn as plane trees, their leaves are drawn on two parallel lines, and only the matching edges may cross, as in Figure~\ref{fig:1cc}. The {\it crossing number} of a layout is the number of unordered 
pairs of matching edges which cross in the drawing.
The {\it tangle crossing number} $\Crt(\T)$ of a tanglegram $\T$ is the minimum crossing number of its layouts. 
Tanglegrams are used in bioinformatics to model co-evolution \cite{Burt,HafnerNadler}.
The tangle crossing number correlates with 
parameters of interests (e.g. number of times parasite switched host, number of horizontal gene transfers) in 
various models of co-evolution and consequently is much studied. Computing the tangle crossing number is NP-hard \cite{buchin, fernau}, but is Fixed Parameter Tractable \cite{buchin, bocker}. It does not allow constant factor approximation under the Unique Game Conjecture \cite{buchin}. Several heuristics for the problem are compared in
\cite{nburg}.
As the reality of biology is infinitely more complicated than this model, \cite{baumann} extended the study of the problem for non-binary trees, \cite{scornavacca}  for phylogenetic networks,
\cite{bansal} for models where correspondence among hosts and parasites is no longer one-to-one.

A tanglegram $\T$ is {\it $m$-crossing-critical} if $\Crt(\T)\geq m$ but any proper induced subtanglegram has crossing number below $m$.
Czabarka, Sz\'ekely and Wagner \cite{tanglekurat} have proved a tanglegram analogue of the Kuratowski Theorem,  the  Tanglegram Kuratowski Theorem: 
\begin{theorem}\label{th:tkurat}
The  1-crossing-critical tanglegrams are
$\K_1$ and $\K_2$.  Hence 
for a tanglegram $\T$,
$\Crt(\T)\ge 1$ if and only if $\T$ contains $\K_1$ or $\K_2$ (see {\rm Figure~\ref{fig:1cc}}) as an induced subtanglegram. 
 \begin{figure}[http]
\centering 
\begin{tikzpicture}
\node[vertex] (lroot) at (-.5,0) {};
\node[vertex] (la1) at (.5,.5) {};
\node[vertex] (la2) at (.5,-.5) {};
\node[vertexb] (ll1) at (1, .75) {};	
\node[vertexb] (ll2) at (1, 0.25) {};		
\node[vertexb] (ll3) at (1, -0.25) {};		
\node[vertexb] (ll4) at (1, -.75) {};		
\draw[thickedge] (la1)--(lroot)--(la2);
\draw[thickedge] (ll1)--(la1)--(ll2);
\draw[thickedge] (ll3)--(la2)--(ll4);
\draw[thickedge] (-1.25,0)--(lroot);
\node[vertex] (rroot) at (3.5,0) {};
\node[vertex] (ra1) at (2.5,.5) {};
\node[vertex] (ra2) at (2.5,-.5) {};
\node[vertexb] (rl1) at (2, .75) {};	
\node[vertexb] (rl2) at (2, 0.25) {};		
\node[vertexb] (rl3) at (2, -0.25) {};		
\node[vertexb] (rl4) at (2, -.75) {};		
\draw[thickedge] (ra1)--(rroot)--(ra2);
\draw[thickedge] (rl1)--(ra1)--(rl2);
\draw[thickedge] (rl3)--(ra2)--(rl4);
\draw[thickedge] (4.25,0)--(rroot);
\draw[thickedge] (ll1)--(rl1);
\draw[thickedge] (ll2)--(rl3);
\draw[thickedge] (ll3)--(rl2);
\draw[thickedge] (ll4)--(rl4);
\node at (1.5,-1.25) {The tanglegram $\K_1$};
\end{tikzpicture}
\quad
\begin{tikzpicture}
\node[vertex] (lroot) at (-.5,0) {};
\node[vertex] (la1) at (0,-.25) {};
\node[vertex] (la2) at (.5,-.5) {};
\node[vertexb] (ll1) at (1, .75) {};	
\node[vertexb] (ll2) at (1, 0.25) {};		
\node[vertexb] (ll3) at (1, -0.25) {};		
\node[vertexb] (ll4) at (1, -.75) {};		
\draw[thickedge] (ll1)--(lroot)--(la1);
\draw[thickedge] (ll2)--(la1)--(la2);
\draw[thickedge] (ll3)--(la2)--(ll4);
\draw[thickedge] (-1.25,0)--(lroot);
\node[vertex] (rroot) at (3.5,0) {};
\node[vertex] (ra1) at (3,.25) {};
\node[vertex] (ra2) at (2.5,.5) {};
\node[vertexb] (rl1) at (2, -.75) {};	
\node[vertexb] (rl2) at (2, -0.25) {};		
\node[vertexb] (rl3) at (2, 0.25) {};		
\node[vertexb] (rl4) at (2, .75) {};		
\draw[thickedge] (rl1)--(rroot)--(ra1);
\draw[thickedge] (rl2)--(ra1)--(ra2);
\draw[thickedge] (rl3)--(ra2)--(rl4);
\draw[thickedge] (4.25,0)--(rroot);
\draw[thickedge] (ll1)--(rl4);
\draw[thickedge] (ll2)--(rl2);
\draw[thickedge] (ll3)--(rl3);
\draw[thickedge] (ll4)--(rl1);
\node at (1.5,-1.25) {The tanglegram $\K_2$};
\end{tikzpicture}
\caption{The 1-crossing-critical tanglegrams $\K_1$ and $\K_2$.} 
\label{fig:1cc}
\end{figure}
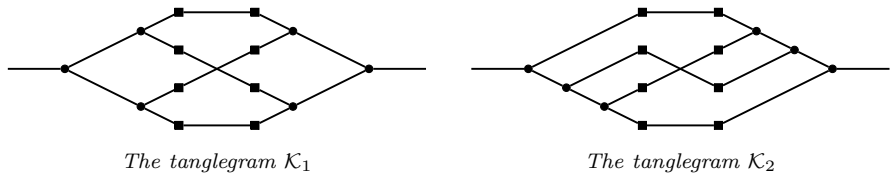
\end{theorem}
Given a tanglegram $\T$, we call an $X\subseteq\sigma_{\T}$ {\it cross-responsible} if $X$ induces a $\K_1$ or $\K_2$ subtanglegram in $\T$.
The Tanglegram Kuratowski Theorem can be reformulated as follows: $\Crt(\T)\ge 1$ if and only if $\T$ has at least one cross-responsible set.

\smallskip
The goal of this paper is to prove the following result:
\begin{theorem}\label{th:main} If a tanglegram $\T$ has a unique cross-responsible set, then $\Crt(\T)=1$. Equivalently, every $2$-crossing-critical
tanglegram contains at least $2$ cross-responsible sets.
\end{theorem}

This would mirror the similar result by Czabarka, Helm and T\'oth which shows that graphs with exactly $k$ cross-inducing subgraphs have crossing number at most $O(k^2)$ and that if $k\in\{0,1,2\}$ that the bound on the crossing number is in fact $k$ itself \cite{geza}.

The proof is shown in two parts: Section~\ref{sec:k1} deals with the case when the cross-responsible set induces a $\K_1$, and Section~\ref{sec:k2}
deals with the case when the cross-responsible set induces a $\K_2$.

We consider this paper as the first step  in our program of classification of 2-crossing-critical tanglegrams.

It remains an open question whether the tangle crossing number of the class of tanglegrams 
with exactly $k$ cross-responsible sets is bounded for every $k$.

\begin{figure}
\centering 
			\begin{tikzpicture}
			[scale=1.2,
			] 		


	\node[vertex] (r1) at (-.5,0) {};
       	\node[circle,draw]  (u1) at (.5,.5) {$a$}; 
        \node[circle,draw]  (u12) at (1,.25) {$2$};
        \node[circle,draw]  (u11) at (1,.75) {$1$};
        \node[vertexb]  (D) at (1.25,.125) {};
        \node[vertexb]  (C) at (1.25,.375) {};
        \node[vertexb]  (B) at (1.25,.625) {};
        \node[vertexb]  (A) at  (1.25,.875) {};
        \node[vertexb]  (E) at (1.25,-.125) {};
        \node[vertexb]  (F) at (1.25,-.375) {};
        \node[vertexb]  (G) at (1.25,-.625) {};
        \node[vertexb] (H)  at (1.25,-.875) {};
        \node[circle,draw]  (d11) at (1,-.75) {$c$};
        \node[circle,draw]  (d12) at (1,-.25) {$b$};
        \node[circle,draw]  (d1) at (.5,-.5) {$3$}; 
        
	\draw[thickeredge] (u1)--(u11)--(A);
	\draw[thickeredge] (u1)--(u12)--(D);
	\draw[thickeredge] (u12)--(C);
	\draw[thickeredge] (u11)--(B);
	\draw[thickeredge] (u1)--(r1)--(d1);
	\draw[thickeredge] (d1)--(d12)--(E);
	\draw[thickeredge] (d1)--(d11)--(H);
	\draw[thickeredge] (d11)--(G);
	\draw[thickeredge] (d12)--(F);
		
	\node[vertex] (r2) at (3.5,0) {};
       	\node[vertex]  (u2) at (2.5,.5) {}; 
        \node[vertex]  (u22) at (2,.25) {};
        \node[vertex]  (u21) at (2,.75) {};
        \node[vertexb]  (F1) at (1.75,.125) {};
        \node[vertexb]  (C1) at (1.75,.375) {};
        \node[vertexb]  (E1) at (1.75,.625) {};
        \node[vertexb]  (A1) at  (1.75,.875) {};
        \node[vertexb]  (B1) at (1.75,-.123) {};
        \node[vertexb]  (G1) at (1.75,-.375) {};
        \node[vertexb]  (D1) at (1.75,-.625) {};
        \node[vertexb] (H1)  at (1.75,-.875) {};
        \node[vertex]  (d22) at (2,-.75) {};
        \node[vertex]  (d21) at (2,-.25) {};
        \node[vertex]  (d2) at (2.5,-.5) {}; 

	\draw[thickeredge] (E1)--(u21)--(A1);
	\draw (u21)--(u2)--(u22);
	\draw[thickeredge] (F1)--(u22)--(C1);
	\draw (u21)--(E1);
	\draw (u2)--(r2);
	\draw (r2)--(d2);
	\draw (d2)--(d22); 
	\draw[thickeredge] (D1)--(d22)--(H1);
	\draw[thickeredge] (B1)-- (d21)--(G1);
	\draw (d21)--(d2);
	
	\draw[thickeredge] (A)--(A1);
	\draw[thickeredge] (B)--(B1);
	\draw[thickeredge] (C)--(C1);
	\draw[thickeredge] (D)--(D1);
	\draw[thickeredge] (E)--(E1);
	\draw[thickeredge] (F)--(F1);
	\draw[thickeredge] (G)--(G1);
	\draw[thickeredge] (H)--(H1);

\end{tikzpicture}
\caption{A tanglegram that contains a subdivided $K_{3,3}$ such that all vertices of the $K_{3,3}$ are in the left tree. Vertex classes are marked with numbers vs. letters, subdivided edges are bold.}
\label{fig:k33}
\end{figure}
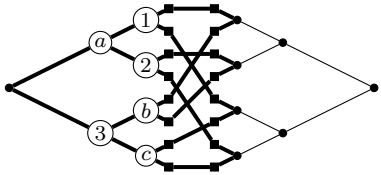

The relative hardness of this question compared to the graph versions is illustrated by the following:
For a tanglegram $T$ we define the associated graph $T^*$ as the graph obtained from $T$ by adding an edge between the roots of the left- and right-tree.
It was shown in \cite{tanglekurat},  $\Crt(T)\ge 1$ precisely when $\Cr(T^*)\ge 1$. As the maximum degree of $T^*$ is $3$, this gives that $\Crt(T)\ge 1$ precisely when $T^*$ contains a subdivided $K_{3,3}$.  Notably, $\K_1^*$ and $\K_2^*$ are subdivided $K_{3,3}$-s. However, it is easy to create a tanglegram $T$ such that $T^*$ 
has subdivided $K_{3,3}$-s that do not correspond to a copy of $\K_1^*$ or $\K_2^*$ -- we give an example on Figure~\ref{fig:k33}. As illustrated on Figure~\ref{fig:connections}, a single  cross-responsible set in a tanglegram may
create any number of copies of $\K_1^*$ (or $\K_2^*$) if the tanglegram has size large enough;  and we also exhibit
a family of tanglegrams with unbounded tangle crossing number such that the rectilinar crossing number of their associated graphs is $1$.

 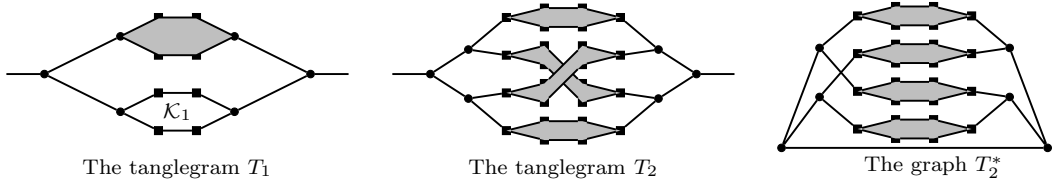
\begin{figure}
\centering 
\begin{tikzpicture}
\node[vertex] (lroot) at (-.5,0) {};
\node[vertex] (lroot1) at (.5,.5) {};
\node[vertexb] (ll1) at (1, .75) {};	
\node[vertexb] (ll2) at (1, 0.25) {};		
\node[vertexb] (rl1) at (1.5, .75) {};		
\node[vertexb] (rl2) at (1.5, 0.25) {};
\node[vertex] (rroot1)	 at (2,.5) {};	
\node[vertex] (rroot)	 at (3,0) {};	
\draw[thickedge,fill=gray!50] (lroot1.center)--(ll1.center)--(rl1.center)--(rroot1.center)--(rl2.center)--(ll2.center)--(lroot1.center);
\node[] at (1.25,-.5) {$\K_1$};
\node[vertex] (lroot2) at (.5,-.5) {};
\node[vertexb] (ll3) at (1, -.75) {};	
\node[vertexb] (ll4) at (1, -0.25) {};		
\node[vertexb] (rl3) at (1.5,-.75) {};		
\node[vertexb] (rl4) at (1.5, -0.25) {};
\node[vertex] (rroot2)	 at (2,-.5) {};	
\draw[thickedge] (lroot2)--(ll3)--(rl3)--(rroot2)--(rl4)--(ll4)--(lroot2);
\draw[thickedge] (lroot1)--(lroot)--(lroot2);
\draw[thickedge] (rroot1)--(rroot)--(rroot2);
\draw[thickedge] (lroot)--(-1,0);
\draw[thickedge] (rroot)--(3.5,0);
\node at (1.25,-1.25) {The tanglegram $T_1$};
\end{tikzpicture}
\quad
\begin{tikzpicture}
\node[vertex] (lroot) at (-.5,0) {};
\node[vertex] (lroot1) at (0,.325) {};
\node[vertexb] (ll1) at (.5, .75) {};	
\node[vertexb] (ll2) at (.5, 0.25) {};		
\node[vertexb] (rl1) at (2, .75) {};		
\node[vertexb] (rl2) at (2, 0.25) {};
\node[vertexb] (la1) at (1,.875) {};
\node[vertexb] (la2) at (1,.625) {};
\node[vertexb] (la3) at (1,.375) {};
\node[vertexb] (la4) at (1,.125) {};
\node[vertexb] (la5) at (1,-.125) {};
\node[vertexb] (la6) at (1,-.375) {};
\node[vertexb] (la7) at (1,-.625) {};
\node[vertexb] (la8) at (1,-.875) {};
\node[vertexb] (ra1) at (1.5,.875) {};
\node[vertexb] (ra2) at (1.5,.625) {};
\node[vertexb] (ra3) at (1.5,.375) {};
\node[vertexb] (ra4) at (1.5,.125) {};
\node[vertexb] (ra5) at (1.5,-.125) {};
\node[vertexb] (ra6) at (1.5,-.375) {};
\node[vertexb] (ra7) at (1.5,-.625) {};
\node[vertexb] (ra8) at (1.5,-.875) {};
\node[vertex] (rroot1)	 at (2.5,.325) {};	
\node[vertex] (rroot)	 at (3,0) {};	
\node[vertex] (lroot2) at (0,-.325) {};
\node[vertexb] (ll3) at (.5, -.25) {};	
\node[vertexb] (ll4) at (.5, -0.75) {};		
\node[vertexb] (rl3) at (2,-.25) {};		
\node[vertexb] (rl4) at (2, -0.75) {};
\node[vertex] (rroot2)	 at (2.5,-.325) {};	
\draw[thickedge] (ll2)--(lroot1)--(ll1);
\draw[thickedge] (rl2)--(rroot1)--(rl1);
\fill[gray!50] (ll1.center)--(la1.center)--(ra1.center)--(rl1.center)--(ra2.center)--(la2.center)--(ll1.center);
\draw[thickedge] (ll1)--(la1)--(ra1)--(rl1)--(ra2)--(la2)--(ll1);
\fill[gray!50] (ll2.center)--(la3.center)--(ra5.center)--(rl3.center)--(ra6.center)--(la4.center)--(ll2.center);
\draw[thickedge] (ll2)--(la3)--(ra5)--(rl3)--(ra6)--(la4)--(ll2);
\fill[gray!50] (rl2.center)--(ra3.center)--(la5.center)--(ll3.center)--(la6.center)--(ra4.center)--(rl2.center);
\draw[thickedge] (rl2)--(ra3)--(la5)--(ll3)--(la6)--(ra4)--(rl2);
\fill[gray!50] (ll4.center)--(la7.center)--(ra7.center)--(rl4.center)--(ra8.center)--(la8.center)--(ll4.center);
\draw[thickedge] (ll4)--(la7)--(ra7)--(rl4)--(ra8)--(la8)--(ll4);
\draw[thickedge] (ll4)--(lroot2)--(ll3);
\draw[thickedge] (rl3)--(rroot2)--(rl4);
\draw[thickedge] (lroot1)--(lroot)--(lroot2);
\draw[thickedge] (rroot1)--(rroot)--(rroot2);
\draw[thickedge] (lroot)--(-1,0);
\draw[thickedge] (rroot)--(3.5,0);
\node at (1.25,-1.25) {The tanglegram $T_2$};
\end{tikzpicture}
\quad
\begin{tikzpicture}
\node[vertex] (lroot) at (-.5,-1) {};
\node[vertex] (lroot1) at (0,.325) {};
\node[vertexb] (ll1) at (.5, .75) {};	
\node[vertexb] (ll3) at (.5, 0.25) {};		
\node[vertexb] (rl1) at (2, .75) {};		
\node[vertexb] (rl2) at (2, 0.25) {};
\node[vertexb] (la1) at (1,.875) {};
\node[vertexb] (la2) at (1,.625) {};
\node[vertexb] (la5) at (1,.375) {};
\node[vertexb] (la6) at (1,.125) {};
\node[vertexb] (la3) at (1,-.125) {};
\node[vertexb] (la4) at (1,-.375) {};
\node[vertexb] (la7) at (1,-.625) {};
\node[vertexb] (la8) at (1,-.875) {};
\node[vertexb] (ra1) at (1.5,.875) {};
\node[vertexb] (ra2) at (1.5,.625) {};
\node[vertexb] (ra3) at (1.5,.375) {};
\node[vertexb] (ra4) at (1.5,.125) {};
\node[vertexb] (ra5) at (1.5,-.125) {};
\node[vertexb] (ra6) at (1.5,-.375) {};
\node[vertexb] (ra7) at (1.5,-.625) {};
\node[vertexb] (ra8) at (1.5,-.875) {};
\node[vertex] (rroot1)	 at (2.5,.325) {};	
\node[vertex] (rroot)	 at (3,-1) {};	
\node[vertex] (lroot2) at (0,-.325) {};
\node[vertexb] (ll2) at (.5, -.25) {};	
\node[vertexb] (ll4) at (.5, -0.75) {};		
\node[vertexb] (rl3) at (2,-.25) {};		
\node[vertexb] (rl4) at (2, -0.75) {};
\node[vertex] (rroot2)	 at (2.5,-.325) {};	
\draw[thickedge] (ll2)--(lroot1)--(ll1);
\draw[thickedge] (rl2)--(rroot1)--(rl1);
\fill[gray!50] (ll1.center)--(la1.center)--(ra1.center)--(rl1.center)--(ra2.center)--(la2.center)--(ll1.center);
\draw[thickedge] (ll1)--(la1)--(ra1)--(rl1)--(ra2)--(la2)--(ll1);
\fill[gray!50] (ll2.center)--(la3.center)--(ra5.center)--(rl3.center)--(ra6.center)--(la4.center)--(ll2.center);
\draw[thickedge] (ll2)--(la3)--(ra5)--(rl3)--(ra6)--(la4)--(ll2);
\fill[gray!50] (rl2.center)--(ra3.center)--(la5.center)--(ll3.center)--(la6.center)--(ra4.center)--(rl2.center);
\draw[thickedge] (rl2)--(ra3)--(la5)--(ll3)--(la6)--(ra4)--(rl2);
\fill[gray!50] (ll4.center)--(la7.center)--(ra7.center)--(rl4.center)--(ra8.center)--(la8.center)--(ll4.center);
\draw[thickedge] (ll4)--(la7)--(ra7)--(rl4)--(ra8)--(la8)--(ll4);
\draw[thickedge] (ll4)--(lroot2)--(ll3);
\draw[thickedge] (rl3)--(rroot2)--(rl4);
\draw[thickedge] (lroot1)--(lroot)--(lroot2);
\draw[thickedge] (rroot1)--(rroot)--(rroot2);
\draw[thickedge] (lroot)--(rroot);
\node at (1.5,-1.25) {The graph $T_2^*$};
\end{tikzpicture}\caption{The gray shaded regions represent a planar tanglegram $F_m$ with $m$ matching edges drawn in a planar way. If two shaded regions
cross, all matching edges between the two copies cross. The tanglegram $T_1$ on the left has a single cross-responsible set that induces a $\K_1$, 
but its associated graph $T_1^*$ has  $m+1$ subdivided $K_{3,3}$-s: the cross-responsible set with any matching edge from $F_m$ induces a 
subtanglegram in $T_1$ that is a subdivision of the graph $\K_1^*$. The middle picture is an optimal layout of tanglegram $T_2$ with tangle crossing number $m^2$, 
and the rightmost picture is an optimal rectilinear drawing of its associated graph $T_2^*$ with exactly one crossing.} 
\label{fig:connections}
\end{figure}

\section{Definitions and basic facts on trees}\label{sec:treedefs}

The following definition of rooted trees slightly differs from the usual, as we require the root to have degree one. Later this helps us to handle scars  in a consistent manner.
\begin{definition} {\em $T$ is a rooted tree} if it is a tree with at least two vertices and a designated root vertex $r_T$ that has degree $1$.
The {\em leaves of $T$} are the non-root vertices  of degree $1$, and the {\it internal vertices} are the vertices of degree at least $2$.
Let $\LL(T)$ denote the set of leaves, and $\II(T)$ denote the set of internal vertices   of $T$.
Given a rooted tree $T$, {\em the tree order $\preceq_T$} is a partial order defined on the vertices of $T$ by 
$x\preceq_T y$ if $x$ lies
on the $r_T$-$y$ path. If $x\preceq_T y$, then we say that {\em $x$ is an ancestor of $y$} or   {\em $y$ is a descendant of $x$      
 in $T$}; and we say that  {\em $x$ is the parent of $y$}, if $x\preceq_T y$ and $x$ is adjacent to $y$. For two vertices $x,y$,  we denote by $x\land_T y$ the largest common lower bound of $x$ and $y$ in $\preceq_T$. 
Two leaves $\lambda_1,\lambda_2$ of $T$ {\em form a cherry} if they have the same parent, i.e., they are at distance $2$ in $T$. An {\em isomorphism} between two rooted trees $T_1,T_2$ is a graph isomorphism between them that maps root to root. If an isomorphism exists, we write $T_1\simeq T_2$.
\end{definition} 

\begin{definition} A {\em  plane  drawing of the rooted tree $T$} is a drawing in the plane, such that 
\begin{itemize}
\item all leaves are incident to a straight line, called 
 the {\em leaf-line},
\item  all non-leaf vertices of the tree are in the same open half-plane of the line, 
\item for any pair of tree vertices $x\ne y$, if $x\preceq_T y$, 
then $y$ is closer to the leaf-line than $x$, 
\item all edges of the tree are drawn in straight line segments, 
\item and the line segments of the edges do not cross. 
\end{itemize}
For a plane drawing $\D$ of the rooted tree $T$, we can speak about the counterclockwise cyclic order of degree 1 vertices, i.e., the leaves and the root of $T$. There is a unique interval in the cyclic order that contains all the leaves but does not contain 
the root. This interval inherits an order from the cyclic order. This order on the leaves is called 
{\em the leaf order $\vec{\ell}_{\D}$ associated with $\D$}  on $\LL(T)$. The leaf order agrees with one of the two 
natural orders of the leaf vertices on the leaf-line  of $\D$.
We deem  two plane  drawings $\D$, $\D^{\prime}$ of the tree $T$  {\em equivalent},  
if $\vec{\ell}_{\D}=\vec{\ell}_{\D^{\prime}}$. 
\end{definition}
For our purposes, there will be no difference between equivalent drawings.

\begin{definition} Let $T$ be a rooted tree with $|\LL(T)|=n$, and let $\vec{\ell}=(\lambda_1,\ldots,\lambda_n)$ be an order of the elements of $\LL(T)$. We call
$\vec{\ell}$  {\em  consistent with $T$}, if there is a plane drawing $\D$ of $T$ such $\vec{\ell}=\vec{\ell}_{\D}$.
\end{definition}

The following Lemma is easy to see and well-known. We prove it for the sole reason that this lemma is a key tool in our proofs.
\begin{lemma}\label{lm:consistent} Let $T$ be a rooted tree and let $\vec{\ell}$ be an order of the elements of $\LL(T)$. 
$\vec{\ell}$ is consistent with $T$ 
if and only if for every $v\in \II(T)$ the set of leaves that are descendants of $v$ form a contiguous segment of $\vec{\ell}$.\\
In particular, if two leaves $x,y$ form a cherry in $T$, then in any consistent order $\vec{\ell}$ of $\LL(T)$ the leaves $x,y$ are next to each other. Moreover, if $\vec{\ell}$ is  consistent with $T$, and  $\vec{\ell}^{\prime}$ is obtained from $\vec{\ell}$ by interchanging $x$ and $y$, then $\vec{\ell}^{\prime}$ is also consistent with $T$. 
\end{lemma}
\begin{proof}
The statement about the leaves that form a cherry follows from the first part, so it is enough to prove that. 
If $\vec{\ell}$ is a consistent ordering of $\LL(T)$, then it is obvious that for any $v\in \II(T)$ the set of leaves that are descendants of $v$ form a contiguous segment of $\vec{\ell}$, so we need to prove the converse of this.
We will use induction on $|\II(T)|$.
If $|\II(T)|\le 1$, then any ordering of the leaves is consistent, so the statement is true. 
Assume that $|\II(T)|=k\ge 2$, and the statement is true for any tree $T^{\star}$ with
$|\II(T^{\star})|<k$. Let $\vec{\ell}$ be an ordering of $\LL(T)$ such that for any $v\in \II(T)$ the set of leaves that are descendants of $v$ form a contiguous segment of $\vec{\ell}$. Let $w_1,\ldots,w_t$ be the children of  $v$, and denote by $T_i$ the tree rooted at $v$ 
obtained from $T$ by removing the root
and all vertices incomparable with $w_i$ in the tree order. It is obvious that a plane drawing of $T$ can be obtained from putting 
standard plane drawings of each $T_i$ in an arbitrary order next to each other on a common leaf line, replacing their root vertex with a common vertex  
above, and connecting that vertex to a root vertex. As in $\vec{\ell}$ the leaves of $T_i$ form a contiguous segment $\vec{\ell}_i$, and  the ordering 
$\vec{\ell}_i$ is consistent by the induction hypothesis and the fact that $|\II(T_i)|<k$,
the statement follows.
\end{proof}
\begin{definition} A {\em rooted binary tree} $T$ is a rooted tree where every internal vertex has degree $3$. 
An edge of $T$ is a leaf-edge (root-edge)  if it is incident to a leaf (to the root).
\end{definition} 

\begin{definition} If $T$ is a rooted binary tree, and $S\subseteq\LL(T)$, {\em the binary subtree $T[S] $ induced by $S$} is obtained as follows:
\begin{itemize}
\item first, let 
$T\llbracket S\rrbracket $ be the rooted subtree of $T$ obtained from $T$ by taking the union of the $r_T$-$s$ paths for all $s\in S$;
\item and then $T[S]$ is obtained from $T\llbracket S\rrbracket $ by suppressing the degree $2$ (subdividing)  vertices. 
\end{itemize}
In this way, every vertex of $T[S]$ can be  identified with a unique  vertex of $T$.
\end{definition}
Note that the induced binary subtree of $T$ is not a subtree in the usual sense. Induced binary subtrees are topical in phylogenetics: if $T$ is the true phylogenetic tree for the taxa present in the leaves,
then the true phylogenetic tree for a subset of those taxa will be the binary subtree induced by that subset of taxa \cite{SempleSteel}.
\begin{definition} \label{rep_path}
Given a binary tree $T^*$, such that      $T[S]\simeq T^{\star}$,  we say that $T\llbracket S\rrbracket $ is a {\em copy} of $T^{\star}$ in $T$.
By the isomorphism, any edge $e$ of $T[S]$ corresponds to a unique {\em representative path} $P_e$ between two vertices of $T\llbracket S\rrbracket $
that are identified with the endvertices of $e$. 
\end{definition}
Next, we put {\em scars} on the edges of $T[S]$. 
\begin{definition} \label{scar}
 In the context of Definition~\ref{rep_path}, we put {\em scars} on the edges of $T[S]$, namely, edge $e$ will get a scar for
every degree 2 vertex of $P_e$. 
Assume that $v\in\LL(T)\setminus S$. We say that a scar on edge $e$ of $T[S]$ was {\em created} for  $v$  if the degree 2 vertex $x$ of $T\llbracket S\rrbracket $ that
corresponds to the scar is the first vertex of the $v$-$r_T$ path starting from $v$  on $P_e$.
\end{definition}
If for $u,v\in\LL(T)\setminus S$, the $u$-$r_T$ and $v$-$r_T$ paths first hit $P_e$ in $T\llbracket S\rrbracket $ in the {\em same} vertex, then they produce the same 
scar on the edge $e$ of $T[S]$, if they first hit  $P_e$ in $T\llbracket S\rrbracket $ in different vertices, then they produce distinct 
scars on the edge $e$ of $T[S]$. If needed, we can recall the leaf set, whose elements created  any particular scar.

We can extend     the tree order $\preceq_{T[S]}$ such that the scars are included  in the underlying set of the tree order, by carrying over the tree order $\preceq_{T\llbracket S\rrbracket}$.

The following lemma is obvious:
\begin{lemma}\label{lm:leafreplace} Let $T$ be a rooted binary tree, $S\subseteq\LL(T)$, with $|S|\ge 2$ and $\lambda_1\in S$. Let $v$ be the parent of $\lambda_1$
 in $T[S]$. If $\lambda_2\in\LL(T)\setminus S$ has its scar in $T[S]$ on one of the edges incident upon $v$, then $T[S]\simeq T[(S\setminus\{\lambda_1\})\cup\{\lambda_2\}]$.
\end{lemma}

Now we are ready to define  {\em induced subdrawings}. The following definition uses the fact that by Lemma~\ref{lm:consistent},
for any planar drawing $\D$ of $\T$, the subsequence formed by the elements of any $S\subseteq\LL(T)$ in $\vec{\ell}_{\D}$ is a consistent order of
$T[S]$.
\begin{definition} Let $T$ be a rooted binary tree, and $S\subseteq\LL(T)$. Let $\D$ be a plane drawing of $T$. 
Then there are natural plane drawings of $T\llbracket S\rrbracket $ and $T[S]$, derived from $\D$. We call these drawings {\em subdrawings}
of $\D$ induced by $S$.
 In these subdrawings, the leaf order
of $S$ is $\vec{\ell}_{\D}$ restricted to $S$.  
 \end{definition}

\section{Definitions and basic facts on tanglegrams}\label{sec:tangdefs}

\begin{definition} A {\em tanglegram $\T$} is a graph, desribed by a triplet $(L_{\T},R_{\T},\sigma_{\T})$, where $L_{\T},R_{\T}$ are rooted binary trees with $|\LL(L_{\T})|=|\LL(R_{\T})|$, and $\sigma_{\T}$ is a perfect matching between $\LL(L_{\T})$ and $\LL(R_{\T})$. $L_{\T}$ and $R_{\T}$ are referred to as the left- and right-trees of $\T$.  Two tanglegrams $\T_1,\T_2$ are {\em isomorphic} if there is a graph isomorphism $f$ between them, such that $f(r_{L_{\T_1}})=r_{L_{\T_2}}$
and $f(r_{R_{\T_1}})=r_{R_{\T_2}}$.
\end{definition}
Induced subtanglegram of a tanglegram is a key concept that corresponds to the concept of induced binary subtree of binary tree.
\begin{definition} Given a tanglegram $\T$ and a set of matching edges $Z\subseteq\sigma_{\T}$,  {\em the subtanglegram $\T[Z]$ induced by $Z$} is defined as follows:
$\sigma_{\T[Z]}=Z$, $L_{\T[Z]}=L_{\T}[S_1]$, and $R_{\T[Z]}=R_{\T}[S_2]$, where
  $S_1\subseteq\LL(L_{\T})$ and $S_2\subseteq\LL(R_{\T})$ are the set of leaves that are matched by $Z$. 
Let  $\T\llbracket Z\rrbracket $ denote the graph $L_{\T}\llbracket S_1\rrbracket\cup R_{\T}\llbracket S_2\rrbracket \cup Z$. If $\T^{\prime}\simeq \T[Z]$, then we say that $\T\llbracket Z\rrbracket$ is a {\em copy} of $\T^{\prime}$ in $\T$. 
\end{definition}
Next we transfer the concept of scar to tanglegrams.
\begin{definition}
Let $\T$ be a tanglegram and $Z\subseteq\sigma_{\T}$. Let $S_L$ and $S_R$ be the endvertices of the edges of $Z$ in $L_{\T}$ and $R_{\T}$ respectively.
 If $m\in\sigma_{\T}\setminus Z$ connects $\lambda_1\in\LL(L_{\T})$ and $\lambda_2\in\LL(R_{\T})$, then we define the {\em left-scar of $m$ in 
 $\T[Z]$} as the scar of $\lambda_1$ in $L_{\T}[S_L]$ as in Definition~\ref{scar}, and we define the {\em right-scar of $m$ in $\T[Z]$}  as the scar of $\lambda_2$ in $R_{\T}[S_R]$ similarly. We call a (left- or right-) scar
{\em an outside scar} if it is on a root-edge, and {\em an inside scar} otherwise.
 We call the {\em scar-type of $m\in\sigma_{\T}\setminus Z$ on $\T[Z]$}   the ordered pair $(e,f)$, if the left-scar of $m$ is on the edge $e$ of $L_{\T}[S_L]$ and the right-scar  of $m$ is on the edge $f$  of $R_{\T}[S_R]$. 
  \end{definition}
In the future, if $X\subset \sigma_T$, such that the endpoint set of $X$ in $L_\T$ is $S_L$, and  the endpoint set of $X$ in $R_\T$ is $S_R$, then we will use the following simplified notation:  $L_{\T}[X]$ for $L_{\T}[S_L]$, and $R_{\T}[X]$ for $R_{\T}[S_R]$.

\begin{definition} Let $\T$ be a tanglegram. A {\em layout of $\T$} is a drawing $\D$ where each edge is drawn as a straight line segment, the trees $L_{\T}$ and $R_{\T}$ are drawn with plane drawings, such that their leaf lines are parallel, and the two trees are
on the outside of the strip enclosed by the leaf lines. 
{\em A representation of a layout $\D$} of the tanglegram $\T$ is the pair $(\vec{\ell}_{L,\D},\vec{\ell}_{R,\D})$ of leaf-orders associated with the plane drawings of $L_{\T}$ and $R_{\T}$ in $\D$.
  Two layouts are  considered {\em equivalent} if they have the same representation. \\
  A layout of a tanglegram naturally defines a {\em sublayout} of every induced subtanglegram of it.
 {\em The crossing number $\Cr(\D)$ of the layout $\D$} is the number of unordered edge pairs that cross in the drawing. The {\em  tangle crossing number $\Crt(\T)$} of the tanglegram $\T$ is the minimum crossing number over all of its layouts. An {\em optimal layout} of $\T$ is a layout $\D$ with 
 $\Cr(\D)=\Crt(\T)$. 
\end{definition}

By definition, only matching edges can cross in a layout, and at most once, furthermore,  equivalent layouts have the same crossing number.
Moreover, the crossing number of a layout can be computed from its representation (note that the directions on the parallel leaf lines are  the opposite
direction for a viewer):
\begin{equation}\Cr(\D)
=\left\vert\left\{ \{\lambda_a\lambda_b,\lambda_c\lambda_d\}: \lambda_a\lambda_b,\lambda_c\lambda_d\in\sigma_{\T}, 
\lambda_a<_{\vec{\ell}_{L,\D}} \lambda_c, \lambda_b<_{\vec{\ell}_{R,\D}} \lambda_d\right\}.
\right\vert
\label{eq:crosslayout}
\end{equation}

\begin{definition} Let $\T$ be a tanglegram.
{\em $X\subseteq\sigma_{\T}$ is crossing-responsible} if $\T[X]\simeq \K_1$ or $\T[X]\simeq\K_2$. Note that if $X$ is cross-responsible, $|X|=4$.
 The collection of crossing-responsible 
edge-sets of $\T$ is denoted by $\XX_{\T}$. 
\end{definition}
 The following lemma is an obvious consequence of Lemma~\ref{lm:leafreplace}:
 \begin{lemma}\label{lm:edgereplace}
 Let $\T$ be a tanglegram, $X\subseteq\sigma_{\T}$ and $m_1\in X$ and $m_2\in\sigma\setminus X$. 
 Assume that $m_1$ connects
 the leaves $\lambda_1\in \LL(L_{\T})$ and $\lambda_2\in\LL(R_{\T})$. Let $G=\{g_1,g_2,g_3\}$ be the set of three edges incident upon the parent
 of $\lambda_1$ in $L_{\T[X]}$ and $H=\{h_1,h_2,h_3\}$ be the set three edges incident upon the parent of $\lambda_2$ in $R_{\T[X]}$. 
 If the scar-type of $m_2$ on $\T[X]$ is in the set $G\times H$, then $\T[(X\setminus\{m_1\})\cup\{m_2\}]\simeq\T[X]$.
 \end{lemma}

\begin{definition}
Let $\T$ be a tanglegram. Two matching edges $e,f\in\sigma$ ($e\ne f$) is called a {\em safe pair in $\T$} if $e,f$ are adjacent to two leaves that form a cherry 
in at least one of $L_{\T}$ and $R_{\T}$. Otherwise pair $e,f$ is an {\em unsafe pair in $\T$}.
\end{definition}

We can now observe
\begin{lemma} If $e,f\in\sigma_{\T}$ is a safe pair in the tanglegram $\T$, then $e,f$ does not cross  in any optimal layout of $\T$. 
\end{lemma}

\begin{proof} Consider a layout $\D$ of $\T$. Assume to the contrary that $e,f$ is a safe pair, and they cross in $\D$. Let $\lambda_1,\lambda_2$
be the leaves that form a cherry and are incident upon $e$ and $f$. Let $\D^{\star}$ be the layout that we obtain from $\D$ by switching the leaf-order of
$\lambda_1,\lambda_2$ in $\vec{\ell}_{L,\D}$ or $\vec{\ell}_{R,\D}$, depending on in which tree they are. (Note that by Lemma~\ref{lm:consistent}, $\lambda_1,\lambda_2$ are next to each other in the leaf-order). 
The only edge-pair whose crossing status differs in $\D$ and $\D^{\star}$ is $e,f$, and therefore
$\D^{\star}$ has a strictly smaller crossing number than $\D$.
\end{proof}

Note that there are two sets of unsafe pairs of matching edges in $\K_1$, and the two sets are disjoint, and there are four sets of unsafe pairs in $\K_2$, which are not disjoint, and for any unsafe pair of $\K_1$ or $\K_2$, the tanglegram has an optimal drawing where only that pair crosses. (This is not true  for arbitrary tanglegrams.) 

\section{When the cross-responsible set induces \texorpdfstring{$\K_1$}{K1}}\label{sec:k1}

In this section we will prove Theorem~\ref{th:main} for tanglegrams, in which the cross-responsible set induces a $\K_1$. The strategy is determining the possible scar-types  on $\T[X]$ of the edges not in the cross-responsible set, and then using them to create a layout of the tanglegram with just one crossing.
 
\begin{lemma}\label{lm:k1notinside} Let $\T$ be a tanglegram with $\XX_{\T}=\{X\}$. Assume that $\T[X]\simeq\K_1$, and let $m\in\sigma_{\T}\setminus X$. Then $m$ has at least one outside scar in $T[X]$.
\end{lemma}

\begin{figure}[http]
\centering 
\begin{tikzpicture}
\node[vertex] (lroot) at (-.5,0) {};
\node[vertex] (la1) at (.5,.5) {};
\node[vertex] (la2) at (.5,-.5) {};
\node[vertex] (lnew) at (-1,0) {};
\node[vertexb] (ll1) at (1, .75) {};	
\node[vertexb] (ll2) at (1, 0.25) {};		
\node[vertexb] (ll3) at (1, -0.25) {};		
\node[vertexb] (ll4) at (1, -.75) {};
\draw[thickedge] (la1)--(lroot)--(la2);
\draw[thickedge] (ll1)--(la1);
\draw[thickedge] (ll3)--(la2);
\draw[thickedge] (la2)--(ll4);
\draw[thickedge] (-1.25,0)--(lroot);
\node[vertex] (rroot) at (3.5,0) {};
\node[vertex] (ra1) at (2.5,.5) {};
\node[vertex] (ra2) at (2.5,-.5) {};
\node[vertex] (rnew) at (2.25,.625) {};
\node[vertexb] (rl1) at (2, .75) {};	
\node[vertexb] (rl2) at (2, 0.25) {};		
\node[vertexb] (rl3) at (2, -0.25) {};		
\node[vertexb] (rl4) at (2, -.75) {};		
\draw[thickedge] (ra1)--(rroot)--(ra2);
\draw[thickedge] (rl1)--(ra1)--(rl2);
\draw[thickedge] (ra2)--(rl4);
\draw[thickedge] (4.25,0)--(rroot);
\draw[thickedge] (ll1)--(rl1);
\draw[thickedge] (ll3)--(rl2);
\draw[thickedge] (ll4)--(rl4);
\node at (0,0.5) {$p$};
\node at (3,-0.5) {$q$};
\node at (.5,.7) {$v_1$};
\node at (2.5,-0.7) {$v_2$};
\draw[thickedge, dashed] (0,0.28)--(1,0)--(2,0)--(3,-0.28);
\node[vertex] (e) at (0.05,0.28) {};
\node[vertex] (f) at (2.95,-0.28) {};
\draw[thickedge, dotted] (la1)--(ll2)--(rl3)--(ra2);
\end{tikzpicture} 
\caption{Example for Lemma~\ref{lm:k1notinside}: adding the edge with scars at $p$ and $q$ (dashed line) and removing the matching edge on the length $3$ path connecting $v_1$ and $v_2$ results in a copy of $\K_1$.} 
\label{fig:k1notleaf}
\end{figure}
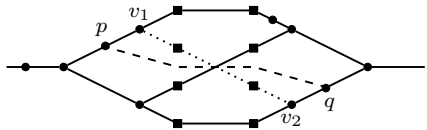 
 \begin{proof} As $\T[X]\simeq\K_1$, every  non-root edge of  $L_{\T[X]}$ or $R_{\T[X]}$ is incident upon an internal vertex with a leaf-neighbor.
Assume to the contrary that the left-scar of $m$ is on the edge $p$ of $L_{\T}[X]$ and the right scar is on the edge $q$ of $R_{\T}[X]$, and $p,q$ are not root-edges.
Let the vertices $v_1\in\II(L_{\T})$ and $v_2\in\II(R_{\T})$ be the internal vertices with a leaf-neighbor that are incident upon $p$ and $q$ respectively. There is a unique matching edge $r\in X$ on the shortest $v_1$-$v_2$ path in $\T[X]$. By Lemma~\ref{lm:edgereplace} 
$\T[(X\setminus\{r\})\cup\{m\}]\simeq\K_1$, which contradicts $\XX_{\T}=\{X\}$.
\end{proof}
 
 \begin{lemma}\label{lm:k1notleaf} Let $\T$ be a tanglegram with $\XX_{\T}=\{X\}$. Assume $\T[X]\simeq\K_1$ and let $m\in\sigma\setminus X$. 
Then none of the scars of $m$ is on a leaf-edge of  $\T[X]$.
 \end{lemma}
  \begin{proof}
 By Lemma~\ref{lm:k1notinside} one of the scars of $m$ (without loss of generality the left-scar) is an outside scar. Assume for the contrary that
 the right-scar of $m$ is on a leaf-edge incident upon the leaf $\lambda_1$ and let $p_1\in X$ be the matching edge incident upon $\lambda_1$. Let $p_2\in X$ be the edge that forms an unsafe pair with $p_1$. Then $(X\setminus\{p_2\})\cup \{m\}$ induces a $\K_2$, a contradiction. 
 
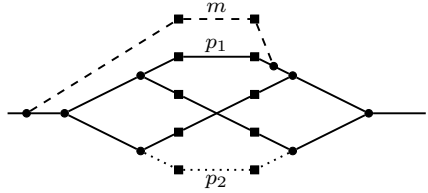
\begin{center}
  \begin{tikzpicture}
\node[vertex] (lroot) at (-.5,0) {};
\node[vertex] (la1) at (.5,.5) {};
\node[vertex] (la2) at (.5,-.5) {};
\node[vertex] (lnew) at (-1,0) {};
\node[vertexb] (ll5) at (1, 1.25) {};	
\node[vertexb] (ll1) at (1, .75) {};	
\node[vertexb] (ll2) at (1, 0.25) {};		
\node[vertexb] (ll3) at (1, -0.25) {};		
\node[vertexb] (ll4) at (1, -.75) {};		
\draw[thickedge] (la1)--(lroot)--(la2);
\draw[thickedge] (ll1)--(la1)--(ll2);
\draw[thickedge] (ll3)--(la2);
\draw[thickedge, dotted] (la2)--(ll4);
\draw[thickedge] (-1.25,0)--(lroot);
\node[vertex] (rroot) at (3.5,0) {};
\node[vertex] (ra1) at (2.5,.5) {};
\node[vertex] (ra2) at (2.5,-.5) {};
\node[vertex] (rnew) at (2.25,.625) {};
\node[vertexb] (rl5) at (2, 1.25) {};	
\node[vertexb] (rl1) at (2, .75) {};	
\node[vertexb] (rl2) at (2, 0.25) {};		
\node[vertexb] (rl3) at (2, -0.25) {};		
\node[vertexb] (rl4) at (2, -.75) {};		
\draw[thickedge] (ra1)--(rroot)--(ra2);
\draw[thickedge] (rl1)--(ra1)--(rl2);
\draw[thickedge] (rl3)--(ra2);
\draw[thickedge,dotted] (ra2)--(rl4);
\draw[thickedge] (4.25,0)--(rroot);
\draw[thickedge] (ll1)--(rl1);
\draw[thickedge] (ll2)--(rl3);
\draw[thickedge] (ll3)--(rl2);
\draw[thickedge, dotted] (ll4)--(rl4);
\node at (1.5,-.9) {$p_2$};
\node at (1.5,.9) {$p_1$};
\draw[thickedge,dashed] (lnew)--(ll5)--(rl5)--(rnew);
\node at (1.5,1.4) {$m$};
\end{tikzpicture}
\end{center}
 \end{proof}
 
 Now we are ready to prove the theorem in the case $\T[X]\simeq \K_1$.
 \begin{theorem}\label{th:k1} Let $\T$ be a tanglegram with $\XX_{\T}=\{X\}$. If $\T[X]\simeq \K_1$, then $\Crt(\T)=1$.
  \end{theorem}
 
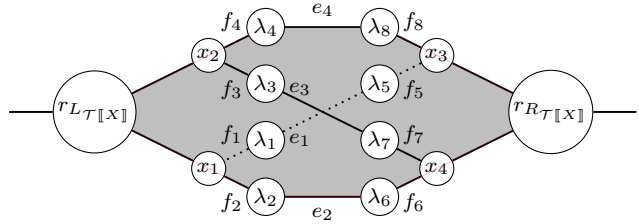
\begin{figure}[http]
\centering 
\begin{tikzpicture}[scale=1.5]
\fill[gray!50] (3.5,0)--(2,.75)--(1,.75)--(-.5,0)--(1,-.75)--(2,-.75)--(3.5,0);			
\node[vertexc] (lroot) at (-.5,0) {$r_{L_{\T\llbracket X\rrbracket}}$};
\node[vertexc] (la1) at (.5,.5) {$x_2$};
\node[vertexc] (la2) at (.5,-.5) {$x_1$};
\node[vertexc] (ll1) at (1, .75) {$\lambda_4$};	
\node[vertexc] (ll2) at (1, 0.25) {$\lambda_3$};		
\node[vertexc] (ll3) at (1, -0.25) {$\lambda_1$};		
\node[vertexc] (ll4) at (1, -.75) {$\lambda_2$};		
\draw[thickedge] (la1)--(lroot)--(la2);
\draw[thickedge] (ll1)--(la1)--(ll2);
\draw[thickedge,dotted] (ll3)--(la2);
\draw[thickedge] (la2)--(ll4);
\draw[thickedge] (-1.25,0)--(lroot);
\node[vertexc] (rroot) at (3.5,0) {$r_{R_{\T\llbracket X\rrbracket}}$};
\node[vertexc] (ra1) at (2.5,.5) {$x_3$};
\node[vertexc] (ra2) at (2.5,-.5) {$x_4$};
\node[vertexc] (rl1) at (2, .75) {$\lambda_8$};	
\node[vertexc] (rl2) at (2, 0.25) {$\lambda_5$};		
\node[vertexc] (rl3) at (2, -0.25) {$\lambda_7$};		
\node[vertexc] (rl4) at (2, -.75) {$\lambda_6$};		
\draw[thickedge] (ra1)--(rroot)--(ra2);
\draw[thickedge] (rl1)--(ra1);
\draw[thickedge,dotted] (ra1)--(rl2);
\draw[thickedge] (rl3)--(ra2);
\draw[thickedge] (ra2)--(rl4);
\draw[thickedge] (4.25,0)--(rroot);
\draw[thickedge] (ll1)--(rl1);
\draw[thickedge] (ll2)--(rl3);
\draw[thickedge,dotted] (ll3)--(rl2);
\draw[thickedge] (ll4)--(rl4);
\node at (1.5,.9) {$e_4$};
\node at (1.5,-.9) {$e_2$};
\node at (1.3,-.25) {$e_1$};
\node at (.7,-.2) {$f_1$};
\node at (.7,-.8) {$f_2$};
\node at (.7,.8) {$f_4$};
\node at (.7,.2) {$f_3$};
\node at (2.3,.2) {$f_5$};
\node at (2.3,.8) {$f_8$};
\node at (2.3,-.2) {$f_7$};
\node at (2.3,-.8) {$f_6$};
\node at (1.3,.2) {$e_3$};
\fill[red] (rroot)--(rl1)--(ll1)--(lroot)--(ll4)--(rl4)--(rroot);
\end{tikzpicture}
\caption{Illustration for Theorem~\ref{th:k1}.} 
\label{fig:thk1}
\end{figure}

 \begin{proof} 
 Since $\XX_{\T}\ne\emptyset$, $\Crt(\T)\ge 1$. To prove the theorem, it is sufficient to find a layout of $\T$ with only one crossing.
  
 Let $X=\{e_1,e_2,e_3,e_4\}$, where the indices are chosen  such that $e_1,e_3$ form an unsafe pair in $\T[X]$
 and the left leaves of $e_1$ and $e_2$ form a cherry in $\T[X]$. Choose the leaves 
 $\lambda_1,\lambda_2,\lambda_3,\lambda_4\in L_{\T}$ and $\lambda_5,\lambda_6,\lambda_7,\lambda_8\in R_{\T}$ such that $e_i=\lambda_{i}\lambda_{4+i}$. 
 Define $x_1=\lambda_1\land_{L_{\T}}\lambda_2$, $x_2=\lambda_3\land_{L_{\T}}\lambda_4$,
  $x_3=\lambda_5\land_{R_{\T}}\lambda_8$, and $x_4=\lambda_6\land_{R_{\T}}\lambda_7$. 
  For each $i$ with $1\le i\le 8$, let $f_i$ be the leaf-edge incident upon $\lambda_i$ in
  $\T[X]$, and $P_{f_i}$ the representative path connecting $\lambda_i$ to the appropriate $x_j$ in the (left- or right-) tree of $\T$ (see Definition~\ref{rep_path}).
  By Lemma~\ref{lm:k1notleaf}, $P_{f_i}=f_i$ (see Figure~\ref{fig:thk1}),  $\{\lambda_1,\lambda_2\},\{\lambda_3,\lambda_4\}$ are both cherries in
  $L_{\T}$ and $\{\lambda_5,\lambda_8\},
  \{\lambda_6,\lambda_7\}$ are both cherries in $R_{\T}$.
  
Consider $\T^{\prime}=\T[\sigma_{\T}\setminus\{e_1\}]$. Since $\XX_{\T^{\prime}}=\emptyset$,  Theorem~\ref{th:tkurat} yields $\Crt(\T^{\prime})=0$. 
Consider a planar layout $\D^{\prime}$ of $\T^{\prime}$ with representation $(\vec{\ell}_{L,\D^{\prime}},\vec{\ell}_{R,D^{\prime}})$. 
$\D^{\prime}$ (or its mirror image to the $x$ axis)
induces a planar sublayout of $\T[X\setminus\{e_1\}]$, where $e_2$ is below $e_4$; the only such layout is shown on Figure~\ref{fig:thk1}. From now on, this is the layout that we call $\D^{\prime}$. 


Consider an arbitrary matching edge $m\in\sigma_{\T}\setminus X$. By Lemma~\ref{lm:k1notinside}, $m$ has at least one 
outside scar in $\T[X]$. The other scar of $m$ cannot be in the interior of the finite domain 
which is bounded by edges $e_2,e_4$, the path $\lambda_2$-$r_{L_{\llbracket X\rrbracket}}$-$\lambda_4$ in $L_{\T^{\prime}}$ and
 the path $\lambda_6$-$r_{R_{\T\llbracket X\rrbracket}}$-$\lambda_8$ in $R_{\T^{\prime}}$, colored gray in Figure~\ref{fig:thk1}. 
 As the boundary of the gray domain
 is described by a simple closed curve, $m$ would intersect this boundary, contradicting the planarity of the drawing 
$\D^{\prime}$. Therefore
 we have that $\lambda_2,\lambda_3,\lambda_4$ is a contiguous subsequence of $\vec{\ell}_{L,\D^{\prime}}$, and 
 $\lambda_8,\lambda_7,\lambda_6$ is a contiguous subsequence of $\vec{\ell}_{R,\D^{\prime}}$. Furthermore,
 as $\D^{\prime}$ is planar, for each $\lambda_a\lambda_b\in\sigma_{\T^{\prime}},\lambda_c\lambda_d\in\sigma_{\T^{\prime}}$,
 where $\lambda_a,\lambda_c$ are leaves in the left tree and $\lambda_b,\lambda_d$ are leaves in the right tree,  
$\lambda_a<_{\vec{\ell}_{L,\D^{\prime}}} \lambda_c$ implies $\lambda_d<_{\vec{\ell}_{R,\D^{\prime}}} \lambda_b$. Now we are in the position to give the planar layout of $\T$ explicitly. We declare the representation of the layout as 
$(\vec{\ell}_L, \vec{\ell}_R)$, where $\vec{\ell}_L$ is obtained from $\vec{\ell}_{L,\D^{\prime}}$ by inserting $\lambda_1$ between $\lambda_2$ and $\lambda_3$, and $\vec{\ell}_R$ is obtained from $\vec{\ell}_{R,\D^{\prime}}$ by  inserting $\lambda_5$ between $\lambda_7$ and
 $\lambda_8$. 
 
It suffices  to check that $\vec{\ell}_L$ and $\vec{\ell}_R$ are 
 consistent leaf-orders of $L_{\T}$ and $R_{\T}$. We are going to 
 show the details for $\vec{\ell}_L$, as the other case is similar. 
   By  Lemma~\ref{lm:consistent}, we just need to show that the leaf-descendants of an 
 arbitrary $v\in\II(L_{\T})$ form a contiguous subsequence
 of $\vec{\ell}_{L}$. If $v=x_1$, the leaf-descendants are its children, $\lambda_1$ and $\lambda_2$, and the statement 
 follows from the construction of $\vec{\ell}_L$.

If $v\ne x_1$, then $v$ is an ancestor of $\lambda_1$ in $\T$ if and only if $v$ is an
 ancestor of $\lambda_2$. Furthermore, in this case   $v\in \II(L_{\T^{\prime}})$, as $x_1$ is the only  vertex of
$\II(L_{\T})$   that we lost in $\II(L_{\T^{\prime}})$.  The leaf-descendants of the children of $v$ in $L_{\T^{\prime}}$
 form a contiguous subsequence of $\vec{\ell}_{L,\D^{\prime}}$.  This contiguous subsequence includes $\lambda_2$ if and only if $\lambda_1$ is a descendant of $v$, and hence
 $\vec{\ell}_L$ is consistent.

 By (\ref{eq:crosslayout}) and the fact that   $\D^{\prime}$ is planar, to find out the crossing number of our layout boils down to counting the number of edges
 in  $\sigma_{\T}\setminus\{e_1\}$ that $e_1$ crosses.

If $m=\lambda_p\lambda_q\in\sigma_{\T}\setminus\{e_1,e_3\}$, then we must have one of the following:
either ($\lambda_p\le_{\vec{\ell}_{L,\D^{\prime}}}\lambda_2$ and $\lambda_6\le_{\vec{\ell}_{R,\D^{\prime}}}\lambda_q$) or
($\lambda_4\le_{\vec{\ell}_{L,\D^{\prime}}}\lambda_p$ and $\lambda_q\le_{\vec{\ell}_{R,\D^{\prime}}}\lambda_8$).
As $\lambda_2\le_{\vec{\ell}_L}\lambda_1\leq_{\vec{\ell}_L}\lambda_4$ and
$\lambda_8\le_{\vec{\ell}_R}\lambda_5\leq_{\vec{\ell}_R}\lambda_6$, we have that $e_1$ and $m$ do not cross in $\D$.
Therefore,
$\Cr(\D)\leq 1$, as $e_1$ still may cross $e_3$ (in fact, it does).
 \end{proof}  
 
\section{When the cross-responsible set induces \texorpdfstring{$\K_2$}{K2}}\label{sec:k2}

In this section we prove Theorem~\ref{th:main} when $\T[X]\simeq\K_2$, essentially using the  strategy used for $\K_1$  in the previous section.
As the only tanglegram automorphism of $\K_2$ is the identity, we will use a standardized labeling of the edge set of our induced $\K_2$ shown on Figure~\ref{fig:labels}. This labeling is needed  to refer to scar-types. The labeling uses a symmetry: the  unique proper graph automorphism of $\K_2$, which exchanges the right root with the left root (and therefore is not a tanglegram automorphism) fixes edges $x$ and $y$, and exchanges edges labeled by the same letter and different subscripts.

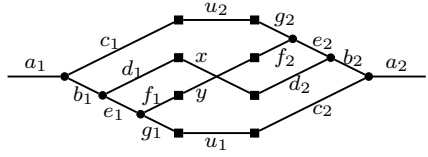
\begin{figure}[http]
\centering 
\begin{tikzpicture}
\node[vertex] (lroot) at (-.5,0) {};
\node[vertex] (la1) at (0,-.25) {};
\node[vertex] (la2) at (.5,-.5) {};
\node[vertexb] (ll1) at (1, .75) {};	
\node[vertexb] (ll2) at (1, 0.25) {};		
\node[vertexb] (ll3) at (1, -0.25) {};		
\node[vertexb] (ll4) at (1, -.75) {};		
\draw[thickedge] (ll1)--(lroot)--(la1);
\node at (.1,.45) {$c_1$};
\node at (-.25,.-.25) {$b_1$};
\draw[thickedge] (ll2)--(la1)--(la2);
\node at (.4,.1) {$d_1$};
\node at (.15,-.5) {$e_1$};
\draw[thickedge] (ll3)--(la2)--(ll4);
\node at (.65,-.75) {$g_1$};
\node at (.65,-.25) {$f_1$};
\draw[thickedge] (-1.25,0)--(lroot);
\node at (-.875,.15) {$a_1$};
\node[vertex] (rroot) at (3.5,0) {};
\node[vertex] (ra1) at (3,.25) {};
\node[vertex] (ra2) at (2.5,.5) {};
\node[vertexb] (rl1) at (2, -.75) {};	
\node[vertexb] (rl2) at (2, -0.25) {};		
\node[vertexb] (rl3) at (2, 0.25) {};		
\node[vertexb] (rl4) at (2, .75) {};		
\draw[thickedge] (rl1)--(rroot)--(ra1);
\node at (2.9,-.45) {$c_2$};
\node at (3.3,.25) {$b_2$};
\draw[thickedge] (rl2)--(ra1)--(ra2);
\node at (2.6,-.1) {$d_2$};
\node at (2.9,.45) {$e_2$};
\draw[thickedge] (rl3)--(ra2)--(rl4);
\node at (2.4,.75) {$g_2$};
\node at (2.4,.25) {$f_2$};
\draw[thickedge] (4.25,0)--(rroot);
\node at (3.875,.15) {$a_2$};
\draw[thickedge] (ll1)--(rl4);
\node at (1.5,.9) {$u_2$};
\draw[thickedge] (ll2)--(rl2);
\node at (1.3,.25) {$x$};
\draw[thickedge] (ll3)--(rl3);
\node at (1.3,-.25) {$y$};
\draw[thickedge] (ll4)--(rl1);
\node at (1.5,-.9) {$u_1$};
\end{tikzpicture}
\caption{Standardized edge-labeling of $\K_2$.} 
\label{fig:labels}
\end{figure}

 \begin{lemma}\label{lm:k2possible}
Let $\T$ be a tanglegram such that $\XX_{\T}=\{X\}$, $\T[X]\simeq\K_2$ and assume $m\in\sigma_{\T}\setminus X$. We have that:
\begin{itemize}
\item The scar type of $m$ in $\T[X]$ is one of the
following: \\
$(a_1,a_2)$, $(a_1,b_2)$, $(b_1,a_2)$, $(a_1,c_2)$, $(c_1,a_2)$, $(b_1,c_2)$, $(c_1,b_2)$, $(d_1,f_2)$, $(f_1,d_2)$.
\item The edges $e_1,e_2,g_1,g_2$ have no scars in $\T[X]$.
\item An edge $m\in\sigma_{\T}\setminus X$ has a scar on $d_j$ in $\T[X]$ precisely when it has a scar on $f_{3-j}$.
\end{itemize}
\end{lemma} 
 
 \begin{proof}
 Scars of $m$ will refer to scars in $\T[X]$.
 Assume that $m\in\sigma\setminus X$ has at least one scar in $\T[X]$ on an inside edge.
By Lemma~\ref{lm:edgereplace} 
the scar-type of $m$ in $\T[X]$ can not be in any of the sets
$\{a_1,b_1,c_1\}\times\{e_{2},f_{2},g_{2}\}$, $\{e_1,  f_1, g_1\}\times\{a_2,b_2,c_2\}$,
$\{b_1,d_1,e_1\}\times\{b_2,d_2,e_2\}$,
$\{e_1,f_1,g_1\}\times\{e_2,f_2,g_2\}$, otherwise we find a second cross-responsible set.

 Therefore to prove the lemma, we only need to show that $m$ can not have scar type $(c_1,c_2)$, and if
 one of the scars is on $d_i$, then the other cannot be on either of $a_{3-i}$, $c_{3-i}$ or $g_{3-i}$. 
 Notice that 
\begin{itemize}
\item if the scar-type of $m$ were $(c_1,c_2)$, then $\{u_1,u_2,y,m\}$ would induce a copy of $\K_1$:
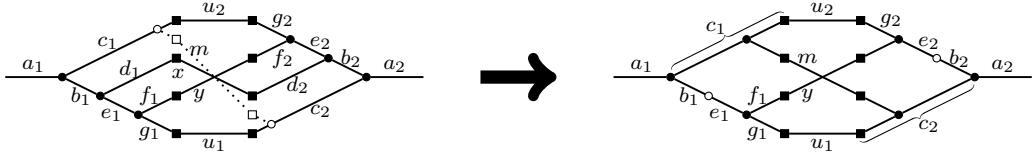
\begin{center}
\begin{tikzpicture}
\node[vertex] (lroot) at (-.5,0) {};
\node[vertex] (la1) at (0,-.25) {};
\node[vertex] (la2) at (.5,-.5) {};
\node[vertexb] (ll1) at (1, .75) {};	
\node[vertexb] (ll2) at (1, 0.25) {};		
\node[vertexb] (ll3) at (1, -0.25) {};		
\node[vertexb] (ll4) at (1, -.75) {};
\node[vertexd] (ll5) at (1, .5) {};		
\draw[thickedge] (ll1)--(lroot)--(la1);
\node at (.1,.45) {$c_1$};
\node at (-.25,.-.25) {$b_1$};
\draw[thickedge] (ll2)--(la1)--(la2);
\node at (.4,.1) {$d_1$};
\node[vertexc] (la3) at (.75,0.635) {};
\node at (.15,-.5) {$e_1$};
\draw[thickedge] (ll3)--(la2)--(ll4);
\node at (.65,-.75) {$g_1$};
\node at (.65,-.25) {$f_1$};
\draw[thickedge] (-1.25,0)--(lroot);
\node at (-.875,.15) {$a_1$};
\node[vertex] (rroot) at (3.5,0) {};
\node[vertex] (ra1) at (3,.25) {};
\node[vertex] (ra2) at (2.5,.5) {};
\node[vertexb] (rl1) at (2, -.75) {};	
\node[vertexb] (rl2) at (2, -0.25) {};		
\node[vertexb] (rl3) at (2, 0.25) {};		
\node[vertexb] (rl4) at (2, .75) {};
\node[vertexd] (rl5) at (2, -.5) {};		
\draw[thickedge] (rl1)--(rroot)--(ra1);
\node at (2.9,-.45) {$c_2$};
\node[vertexc] (ra3) at (2.25, -.625) {};	
\node at (3.3,.25) {$b_2$};
\draw[thickedge] (rl2)--(ra1)--(ra2);
\node at (2.6,-.1) {$d_2$};
\node at (2.9,.45) {$e_2$};
\draw[thickedge] (rl3)--(ra2)--(rl4);
\node at (2.4,.75) {$g_2$};
\node at (2.4,.25) {$f_2$};
\draw[thickedge] (4.25,0)--(rroot);
\node at (3.875,.15) {$a_2$};
\draw[thickedge,dotted] (la3)--(ll5)--(rl5)--(ra3);
\node at (1.3,.35) {$m$};
\draw[thickedge] (ll1)--(rl4);
\node at (1.5,.9) {$u_2$};
\draw[thickedge] (ll2)--(rl2);
\node at (1.025,.05) {$x$};
\draw[thickedge] (ll3)--(rl3);
\node at (1.3,-.25) {$y$};
\draw[thickedge] (ll4)--(rl1);
\node at (1.5,-.9) {$u_1$};
\draw[->,line width=5pt] (5,0)--(6,0);
\node[vertex] (Lroot) at (7.5,0) {};
\node[vertex] (La1) at (8.5,.5) {};
\node[vertex] (La2) at (8.5,-.5) {};
\node[vertexb] (Ll1) at (9, .75) {};	
\node[vertexb] (Ll2) at (9, 0.25) {};		
\node[vertexb] (Ll3) at (9, -0.25) {};		
\node[vertexb] (Ll4) at (9, -.75) {};	
\draw[thickedge] (La1)--(Lroot)--(La2);
\node[vertexc] at (8,-.25) {};
\draw [decorate,
	decoration = {calligraphic brace}] (7.51,.1) --  (8.99,.85);
\node at (8.1,.65) {$c_1$};
\draw[thickedge] (Ll1)--(La1)--(Ll2);
\node at (8.15,-.5) {$e_1$};
\node at (7.75,.-.25) {$b_1$};
\draw[thickedge] (Ll3)--(La2)--(Ll4);
\node at (8.65,-.75) {$g_1$};
\node at (8.65,-.25) {$f_1$};
\draw[thickedge] (6.75,0)--(Lroot);
\node at (7.125,.15) {$a_1$};
\node[vertex] (Rroot) at (11.5,0) {};
\node[vertex] (Ra1) at (10.5,-.5) {};
\node[vertex] (Ra2) at (10.5,.5) {};
\node[vertexb] (Rl1) at (10, -.75) {};	
\node[vertexb] (Rl2) at (10, -0.25) {};		
\node[vertexb] (Rl3) at (10, 0.25) {};		
\node[vertexb] (Rl4) at (10, .75) {};
\draw [decorate,
	decoration = {calligraphic brace}] (11.49,-.1) --  (9.99,-.85);
\draw[thickedge] (Ra2)--(Rroot)--(Ra1);
\node at (10.9,-.65) {$c_2$};
\draw[thickedge] (Rl2)--(Ra1)--(Rl1);
\draw[thickedge] (Rl3)--(Ra2)--(Rl4);
\node at (10.4,.75) {$g_2$};
\node at (11.3,.25) {$b_2$};
\node at (10.9,.45) {$e_2$};
\node[vertexc] at (11,.25) {};
\draw[thickedge] (12.25,0)--(Rroot);
\node at (11.875,.15) {$a_2$};
\draw[thickedge] (Ll1)--(Rl4);
\node at (9.5,.9) {$u_2$};
\draw[thickedge] (Ll2)--(Rl2);
\node at (9.3,.25) {$m$};
\draw[thickedge] (Ll3)--(Rl3);
\node at (9.3,-.25) {$y$};
\draw[thickedge] (Ll4)--(Rl1);
\node at (9.5,-.9) {$u_1$};
\end{tikzpicture}
\end{center}
\item if the scars of $m$ were $d_i$ and $a_{3-i}$ for some $i\in\{1,2\}$, then $\{x,u_i,u_{3-i},m\}$ would induce a copy of $\K_2$:
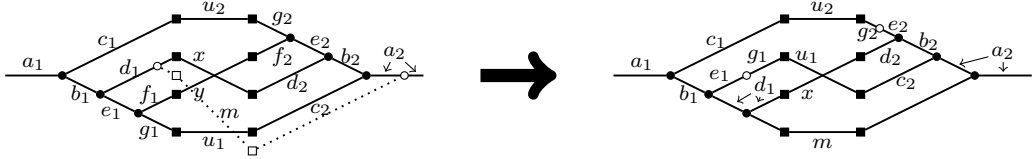
\begin{center}
\begin{tikzpicture} 
\node[vertex] (lroot) at (-.5,0) {};
\node[vertex] (la1) at (0,-.25) {};
\node[vertex] (la2) at (.5,-.5) {};
\node[vertexb] (ll1) at (1, .75) {};	
\node[vertexb] (ll2) at (1, 0.25) {};		
\node[vertexb] (ll3) at (1, -0.25) {};		
\node[vertexb] (ll4) at (1, -.75) {};
\node[vertexd] (ll5) at (1, 0) {};		
\draw[thickedge] (ll1)--(lroot)--(la1);
\node at (.1,.45) {$c_1$};
\node at (-.25,.-.25) {$b_1$};
\draw[thickedge] (ll2)--(la1)--(la2);
\node at (.4,.1) {$d_1$};
\node[vertexc] (la3) at (.75,0.125) {};
\node at (.15,-.5) {$e_1$};
\draw[thickedge] (ll3)--(la2)--(ll4);
\node at (.65,-.75) {$g_1$};
\node at (.65,-.25) {$f_1$};
\draw[thickedge] (-1.25,0)--(lroot);
\node at (-.875,.15) {$a_1$};
\node[vertex] (rroot) at (3.5,0) {};
\node[vertex] (ra1) at (3,.25) {};
\node[vertex] (ra2) at (2.5,.5) {};
\node[vertexb] (rl1) at (2, -.75) {};	
\node[vertexb] (rl2) at (2, -0.25) {};		
\node[vertexb] (rl3) at (2, 0.25) {};		
\node[vertexb] (rl4) at (2, .75) {};
\node[vertexd] (rl5) at (2, -1) {};		
\draw[thickedge] (rl1)--(rroot)--(ra1);
\node at (2.9,-.45) {$c_2$};
\node at (3.3,.25) {$b_2$};
\draw[thickedge] (rl2)--(ra1)--(ra2);
\node at (2.6,-.1) {$d_2$};
\node at (2.9,.45) {$e_2$};
\draw[thickedge] (rl3)--(ra2)--(rl4);
\node at (2.4,.75) {$g_2$};
\node at (2.4,.25) {$f_2$};
\draw[thickedge] (4.25,0)--(rroot);
\node[vertexc] (ra3) at (4, 0) {};	
\node (rre) at (3.875,.3) {$a_2$};
\draw[->] (rre)--(3.75,.05);
\draw[->] (rre)--(4.15,.05);
\draw[thickedge,dotted] (la3)--(ll5)--(rl5)--(ra3);
\node at (1.7,-.5) {$m$};
\draw[thickedge] (ll1)--(rl4);
\node at (1.5,.9) {$u_2$};
\draw[thickedge] (ll2)--(rl2);
\node at (1.3,.25) {$x$};
\draw[thickedge] (ll3)--(rl3);
\node at (1.3,-.25) {$y$};
\draw[thickedge] (ll4)--(rl1);
\node at (1.5,-.9) {$u_1$};
\draw[->,line width=5pt] (5,0)--(6,0);
\node[vertex] (Lroot) at (7.5,0) {};
\node[vertex] (La1) at (8,-.25) {};
\node[vertex] (La2) at (8.5,-.5) {};
\node[vertexb] (Ll1) at (9, .75) {};	
\node[vertexb] (Ll2) at (9, 0.25) {};		
\node[vertexb] (Ll3) at (9, -0.25) {};		
\node[vertexb] (Ll4) at (9, -.75) {};	
\draw[thickedge] (Ll1)--(Lroot)--(La1);
\node at (8.1,.45) {$c_1$};
\node at (7.75,.-.25) {$b_1$};
\draw[thickedge] (Ll2)--(La1)--(La2);
\node (newd1) at (8.75,-.125) {$d_1$};
\draw[->] (newd1)--(8.375,-.35);
\draw[->] (newd1)--(8.65,-.35);
\node[vertexc] (La3) at (8.5,0) {};
\node at (8.15,0) {$e_1$};
\draw[thickedge] (Ll3)--(La2)--(Ll4);
\node at (8.65,.25) {$g_1$};
\draw[thickedge] (6.75,0)--(Lroot);
\node at (7.125,.15) {$a_1$};
\node[vertex] (Rroot) at (11.5,0) {};
\node[vertex] (Ra1) at (11,.25) {};
\node[vertex] (Ra2) at (10.5,.5) {};
\node[vertexb] (Rl1) at (10, -.75) {};	
\node[vertexb] (Rl2) at (10, -0.25) {};		
\node[vertexb] (Rl3) at (10, 0.25) {};		
\node[vertexb] (Rl4) at (10, .75) {};	
\draw[thickedge] (Rl1)--(Rroot)--(Ra1);
\draw[thickedge] (Rl2)--(Ra1)--(Ra2);
\node at (10.6,-.1) {$c_2$};
\node at (10.9,.45) {$b_2$};
\draw[thickedge] (Rl3)--(Ra2)--(Rl4);
\node at (10.1,.55) {$g_2$};
\node at (10.5,.65) {$e_2$};
\node at (10.4,.25) {$d_2$};
\draw[thickedge] (12.25,0)--(Rroot);
\node[vertexc] (Ra3) at (10.25, .625) {};	
\node (rre) at (11.875,.3) {$a_2$};
\draw[->] (rre)--(11.3,.15);
\draw[->] (rre)--(11.875,.05);
\draw[thickedge] (Ll1)--(Rl4);
\node at (9.5,.9) {$u_2$};
\draw[thickedge] (Ll2)--(Rl2);
\node at (9.3,.25) {$u_1$};
\draw[thickedge] (Ll3)--(Rl3);
\node at (9.3,-.25) {$x$};
\draw[thickedge] (Ll4)--(Rl1);
\node at (9.5,-.9) {$m$};
\end{tikzpicture}
\end{center}
\item if the scars of $m$ were $d_i$ and $c_{3-i}$ for some $i\in\{1,2\}$, then $\{x,y,u_i,m\}$ would induce a copy of $\K_1$:
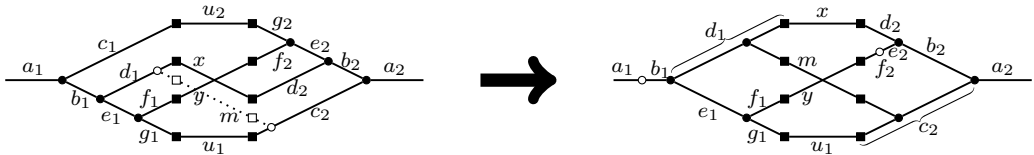
\begin{center}
\begin{tikzpicture} 
\node[vertex] (lroot) at (-.5,0) {};
\node[vertex] (la1) at (0,-.25) {};
\node[vertex] (la2) at (.5,-.5) {};
\node[vertexb] (ll1) at (1, .75) {};	
\node[vertexb] (ll2) at (1, 0.25) {};		
\node[vertexb] (ll3) at (1, -0.25) {};		
\node[vertexb] (ll4) at (1, -.75) {};
\node[vertexd] (ll5) at (1, 0) {};		
\draw[thickedge] (ll1)--(lroot)--(la1);
\node at (.1,.45) {$c_1$};
\node at (-.25,.-.25) {$b_1$};
\draw[thickedge] (ll2)--(la1)--(la2);
\node at (.4,.1) {$d_1$};
\node[vertexc] (la3) at (.75,0.125) {};
\node at (.15,-.5) {$e_1$};
\draw[thickedge] (ll3)--(la2)--(ll4);
\node at (.65,-.75) {$g_1$};
\node at (.65,-.25) {$f_1$};
\draw[thickedge] (-1.25,0)--(lroot);
\node at (-.875,.15) {$a_1$};
\node[vertex] (rroot) at (3.5,0) {};
\node[vertex] (ra1) at (3,.25) {};
\node[vertex] (ra2) at (2.5,.5) {};
\node[vertexb] (rl1) at (2, -.75) {};	
\node[vertexb] (rl2) at (2, -0.25) {};		
\node[vertexb] (rl3) at (2, 0.25) {};		
\node[vertexb] (rl4) at (2, .75) {};
\node[vertexd] (rl5) at (2, -.5) {};		
\draw[thickedge] (rl1)--(rroot)--(ra1);
\node at (2.9,-.45) {$c_2$};
\node[vertexc] (ra3) at (2.25, -.625) {};	
\node at (3.3,.25) {$b_2$};
\draw[thickedge] (rl2)--(ra1)--(ra2);
\node at (2.6,-.1) {$d_2$};
\node at (2.9,.45) {$e_2$};
\draw[thickedge] (rl3)--(ra2)--(rl4);
\node at (2.4,.75) {$g_2$};
\node at (2.4,.25) {$f_2$};
\draw[thickedge] (4.25,0)--(rroot);
\node at (3.875,.15) {$a_2$};
\draw[thickedge,dotted] (la3)--(ll5)--(rl5)--(ra3);
\node at (1.7,-.5) {$m$};
\draw[thickedge] (ll1)--(rl4);
\node at (1.5,.9) {$u_2$};
\draw[thickedge] (ll2)--(rl2);
\node at (1.3,.25) {$x$};
\draw[thickedge] (ll3)--(rl3);
\node at (1.3,-.25) {$y$};
\draw[thickedge] (ll4)--(rl1);
\node at (1.5,-.9) {$u_1$};
\draw[->,line width=5pt] (5,0)--(6,0);
\node[vertex] (Lroot) at (7.5,0) {};
\node[vertex] (La1) at (8.5,.5) {};
\node[vertex] (La2) at (8.5,-.5) {};
\node[vertexb] (Ll1) at (9, .75) {};	
\node[vertexb] (Ll2) at (9, 0.25) {};		
\node[vertexb] (Ll3) at (9, -0.25) {};		
\node[vertexb] (Ll4) at (9, -.75) {};	
\draw[thickedge] (La1)--(Lroot)--(La2);
\draw [decorate,
	decoration = {calligraphic brace}] (7.51,.1) --  (8.99,.85);
\node at (8.1,.65) {$d_1$};
\draw[thickedge] (Ll1)--(La1)--(Ll2);
\node at (8,-.45) {$e_1$};
\draw[thickedge] (Ll3)--(La2)--(Ll4);
\node at (8.65,-.75) {$g_1$};
\node at (8.65,-.25) {$f_1$};
\draw[thickedge] (6.75,0)--(Lroot);
\node[vertexc] at (7.125,0) {};
\node at (7.375,.15) {$b_1$};
\node at (6.875,.15) {$a_1$};
\node[vertex] (Rroot) at (11.5,0) {};
\node[vertex] (Ra1) at (10.5,-.5) {};
\node[vertex] (Ra2) at (10.5,.5) {};
\node[vertexb] (Rl1) at (10, -.75) {};	
\node[vertexb] (Rl2) at (10, -0.25) {};		
\node[vertexb] (Rl3) at (10, 0.25) {};		
\node[vertexb] (Rl4) at (10, .75) {};
\draw [decorate,
	decoration = {calligraphic brace}] (11.49,-.1) --  (9.99,-.85);
\draw[thickedge] (Ra2)--(Rroot)--(Ra1);
\node at (10.9,-.65) {$c_2$};
\node at (11,.45) {$b_2$};
\draw[thickedge] (Rl2)--(Ra1)--(Rl1);
\draw[thickedge] (Rl3)--(Ra2)--(Rl4);
\node at (10.4,.75) {$d_2$};
\node at (10.325,.155) {$f_2$};
\node[vertexc] at (10.25,.375) {};
\node at (10.5, .35) {$e_2$};
\draw[thickedge] (12.25,0)--(Rroot);
\node at (11.875,.15) {$a_2$};
\draw[thickedge] (Ll1)--(Rl4);
\node at (9.5,.9) {$x$};
\draw[thickedge] (Ll2)--(Rl2);
\node at (9.3,.25) {$m$};
\draw[thickedge] (Ll3)--(Rl3);
\node at (9.3,-.25) {$y$};
\draw[thickedge] (Ll4)--(Rl1);
\node at (9.5,-.9) {$u_1$};
\end{tikzpicture}
\end{center}
\item if the scars of $m$ were $d_i$ and $g_{3-i}$ for some $i\in\{1,2\}$, then $\{x,y,u_{3-i},m\}$ would induce a copy of $\K_2$:
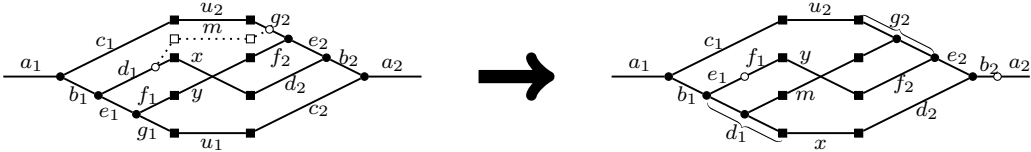
\begin{center}
\begin{tikzpicture} 
\node[vertex] (lroot) at (-.5,0) {};
\node[vertex] (la1) at (0,-.25) {};
\node[vertex] (la2) at (.5,-.5) {};
\node[vertexb] (ll1) at (1, .75) {};	
\node[vertexb] (ll2) at (1, 0.25) {};		
\node[vertexb] (ll3) at (1, -0.25) {};		
\node[vertexb] (ll4) at (1, -.75) {};
\node[vertexd] (ll5) at (1, .5) {};		
\draw[thickedge] (ll1)--(lroot)--(la1);
\node at (.1,.45) {$c_1$};
\node at (-.25,.-.25) {$b_1$};
\draw[thickedge] (ll2)--(la1)--(la2);
\node at (.4,.1) {$d_1$};
\node[vertexc] (la3) at (.75,0.125) {};
\node at (.15,-.5) {$e_1$};
\draw[thickedge] (ll3)--(la2)--(ll4);
\node at (.65,-.75) {$g_1$};
\node at (.65,-.25) {$f_1$};
\draw[thickedge] (-1.25,0)--(lroot);
\node at (-.875,.15) {$a_1$};
\node[vertex] (rroot) at (3.5,0) {};
\node[vertex] (ra1) at (3,.25) {};
\node[vertex] (ra2) at (2.5,.5) {};
\node[vertexb] (rl1) at (2, -.75) {};	
\node[vertexb] (rl2) at (2, -0.25) {};		
\node[vertexb] (rl3) at (2, 0.25) {};		
\node[vertexb] (rl4) at (2, .75) {};
\node[vertexd] (rl5) at (2, .5) {};		
\draw[thickedge] (rl1)--(rroot)--(ra1);
\node at (2.9,-.45) {$c_2$};
\node at (3.3,.25) {$b_2$};
\draw[thickedge] (rl2)--(ra1)--(ra2);
\node at (2.6,-.1) {$d_2$};
\node at (2.9,.45) {$e_2$};
\draw[thickedge] (rl3)--(ra2)--(rl4);
\node at (2.4,.75) {$g_2$};
\node at (2.4,.25) {$f_2$};
\draw[thickedge] (4.25,0)--(rroot);
\node[vertexc] (ra3) at (2.25, .625) {};	
\node at (3.875,.15) {$a_2$};
\draw[thickedge,dotted] (la3)--(ll5)--(rl5)--(ra3);
\node at (1.5,.635) {$m$};
\draw[thickedge] (ll1)--(rl4);
\node at (1.5,.9) {$u_2$};
\draw[thickedge] (ll2)--(rl2);
\node at (1.3,.25) {$x$};
\draw[thickedge] (ll3)--(rl3);
\node at (1.3,-.25) {$y$};
\draw[thickedge] (ll4)--(rl1);
\node at (1.5,-.9) {$u_1$};
\draw[->,line width=5pt] (5,0)--(6,0);
\node[vertex] (Lroot) at (7.5,0) {};
\node[vertex] (La1) at (8,-.25) {};
\node[vertex] (La2) at (8.5,-.5) {};
\node[vertexb] (Ll1) at (9, .75) {};	
\node[vertexb] (Ll2) at (9, 0.25) {};		
\node[vertexb] (Ll3) at (9, -0.25) {};		
\node[vertexb] (Ll4) at (9, -.75) {};	
\draw[thickedge] (Ll1)--(Lroot)--(La1);
\node at (8.1,.45) {$c_1$};
\node at (7.75,.-.25) {$b_1$};
\draw[thickedge] (Ll2)--(La1)--(La2);
\draw [decorate,
	decoration = {calligraphic brace}] (8.99,-.85) --  (8.01,-.35);
\node at (8.4,-.75) {$d_1$};
\node[vertexc] (La3) at (8.5,0) {};
\node at (8.15,0) {$e_1$};
\draw[thickedge] (Ll3)--(La2)--(Ll4);
\node at (8.65,.25) {$f_1$};
\draw[thickedge] (6.75,0)--(Lroot);
\node at (7.125,.15) {$a_1$};
\node[vertex] (Rroot) at (11.5,0) {};
\node[vertex] (Ra1) at (11,.25) {};
\node[vertex] (Ra2) at (10.5,.5) {};
\node[vertexb] (Rl1) at (10, -.75) {};	
\node[vertexb] (Rl2) at (10, -0.25) {};		
\node[vertexb] (Rl3) at (10, 0.25) {};		
\node[vertexb] (Rl4) at (10, .75) {};		
\draw[thickedge] (Rl1)--(Rroot)--(Ra1);
\node at (10.9,-.45) {$d_2$};
\node at (11.3,.25) {$e_2$};
\draw[thickedge] (Rl2)--(Ra1)--(Ra2);
\node at (10.6,-.1) {$f_2$}; 
\draw[thickedge] (Rl3)--(Ra2)--(Rl4);
\draw [decorate,
	decoration = {calligraphic brace}] (10.01,.8) --  (10.99,.3);
\node at (10.55,.75) {$g_2$};
\draw[thickedge] (12.25,0)--(Rroot);
\node[vertexc] (Ra3) at (11.825, 0) {};
\node at (11.725,.15) {$b_2$};	
\node at (12.125,.15) {$a_2$};
\draw[thickedge] (Ll1)--(Rl4);
\node at (9.5,.9) {$u_2$};
\draw[thickedge] (Ll2)--(Rl2);
\node at (9.3,.25) {$y$};
\draw[thickedge] (Ll3)--(Rl3);
\node at (9.3,-.25) {$m$};
\draw[thickedge] (Ll4)--(Rl1);
\node at (9.5,-.9) {$x$};
\end{tikzpicture}
\end{center}
\end{itemize}
We excluded all scar-types disallowed in the Lemma, hence the proof is finished. 
\end{proof}
 
 \begin{lemma}\label{lm:k2dscars}
 Let $\T$ be a tanglegram, such that $\XX_{\T}=\{X\}$, and $\T[X]\simeq\K_2$. For $j\in\{1,2\}$ let
 $M_{j}$ be the set of edges in $\sigma\setminus X$ that create a scar on $d_j$ in $\T[X]$.
 Then at most one of the sets $M_{1}$, $M_{2}$ is non-empty. 
  \end{lemma}
  
 \begin{proof} 
Assume that both $M_{1}$ and $M_{2}$ are non-empty, and
 let $m_1\in M_{1}$ and $m_2\in M_{2}$. By Lemma~\ref{lm:k2possible} the scar-type of $m_1$ is $(d_1,f_2)$ and the scar-type of $m_2$ is 
 $(f_1,d_2)$ in $\T[X]$.
Then $\{m_1,m_2,x,y\}$ induces a $\K_1$, which is a contradiction. 
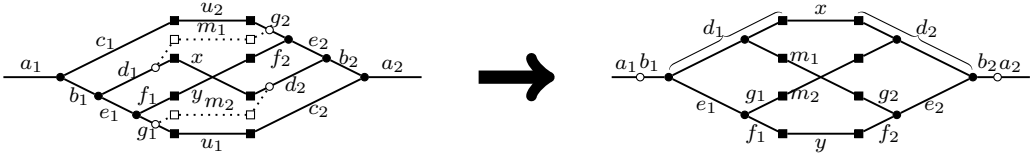
\begin{center}
\begin{tikzpicture} 
\node[vertex] (lroot) at (-.5,0) {};
\node[vertex] (la1) at (0,-.25) {};
\node[vertex] (la2) at (.5,-.5) {};
\node[vertexb] (ll1) at (1, .75) {};	
\node[vertexb] (ll2) at (1, 0.25) {};		
\node[vertexb] (ll3) at (1, -0.25) {};		
\node[vertexb] (ll4) at (1, -.75) {};
\node[vertexd] (ll5) at (1, .5) {};		
\node[vertexd] (ll6) at (1, -.5) {};		
\draw[thickedge] (ll1)--(lroot)--(la1);
\node at (.1,.45) {$c_1$};
\node at (-.25,.-.25) {$b_1$};
\draw[thickedge] (ll2)--(la1)--(la2);
\node at (.4,.1) {$d_1$};
\node[vertexc] (la3) at (.75,0.125) {};
\node at (.15,-.5) {$e_1$};
\draw[thickedge] (ll3)--(la2)--(ll4);
\node[vertexc] (la4) at (.75,-0.625) {};
\node at (.65,-.75) {$g_1$};
\node at (.65,-.25) {$f_1$};
\draw[thickedge] (-1.25,0)--(lroot);
\node at (-.875,.15) {$a_1$};
\node[vertex] (rroot) at (3.5,0) {};
\node[vertex] (ra1) at (3,.25) {};
\node[vertex] (ra2) at (2.5,.5) {};
\node[vertexb] (rl1) at (2, -.75) {};	
\node[vertexb] (rl2) at (2, -0.25) {};		
\node[vertexb] (rl3) at (2, 0.25) {};		
\node[vertexb] (rl4) at (2, .75) {};
\node[vertexd] (rl5) at (2, .5) {};	
\node[vertexd] (rl6) at (2, -.5) {};			
\draw[thickedge] (rl1)--(rroot)--(ra1);
\node at (2.9,-.45) {$c_2$};
\node at (3.3,.25) {$b_2$};
\draw[thickedge] (rl2)--(ra1)--(ra2);
\node at (2.6,-.1) {$d_2$};
\node at (2.9,.45) {$e_2$};
\draw[thickedge] (rl3)--(ra2)--(rl4);
\node at (2.4,.75) {$g_2$};
\node at (2.4,.25) {$f_2$};
\draw[thickedge] (4.25,0)--(rroot);
\node[vertexc] (ra4) at (2.25,-.125) {};
\node[vertexc] (ra3) at (2.25, .625) {};	
\node at (3.875,.15) {$a_2$};
\draw[thickedge,dotted] (la3)--(ll5)--(rl5)--(ra3);
\draw[thickedge,dotted] (la4)--(ll6)--(rl6)--(ra4);
\node at (1.5,.635) {$m_1$};
\node at (1.6,-.350) {$m_2$};
\draw[thickedge] (ll1)--(rl4);
\node at (1.5,.9) {$u_2$};
\draw[thickedge] (ll2)--(rl2);
\node at (1.3,.25) {$x$};
\draw[thickedge] (ll3)--(rl3);
\node at (1.3,-.25) {$y$};
\draw[thickedge] (ll4)--(rl1);
\node at (1.5,-.9) {$u_1$};
\draw[->,line width=5pt] (5,0)--(6,0);
\node[vertex] (Lroot) at (7.5,0) {};
\node[vertex] (La1) at (8.5,.5) {};
\node[vertex] (La2) at (8.5,-.5) {};
\node[vertexb] (Ll1) at (9, .75) {};	
\node[vertexb] (Ll2) at (9, 0.25) {};		
\node[vertexb] (Ll3) at (9, -0.25) {};		
\node[vertexb] (Ll4) at (9, -.75) {};	
\draw[thickedge] (La1)--(Lroot)--(La2);
\draw [decorate,
	decoration = {calligraphic brace}] (7.51,.1) --  (8.99,.85);
\node at (8.1,.65) {$d_1$};
\draw[thickedge] (Ll1)--(La1)--(Ll2);
\node at (8,-.4) {$e_1$};
\draw[thickedge] (Ll3)--(La2)--(Ll4);
\node at (8.65,-.75) {$f_1$};
\node at (8.65,-.25) {$g_1$};
\draw[thickedge] (6.75,0)--(Lroot);
\node[vertexc] at (7.125,0) {};
\node at (6.95,.15) {$a_1$};
\node at (7.255,.2) {$b_1$};
\node[vertex] (Rroot) at (11.5,0) {};
\node[vertex] (Ra1) at (10.5,-.5) {};
\node[vertex] (Ra2) at (10.5,.5) {};
\node[vertexb] (Rl1) at (10, -.75) {};	
\node[vertexb] (Rl2) at (10, -0.25) {};		
\node[vertexb] (Rl3) at (10, 0.25) {};		
\node[vertexb] (Rl4) at (10, .75) {};
\draw [decorate,
	decoration = {calligraphic brace}]  (9.99,.85)--(11.49,.1);
\draw[thickedge] (Ra2)--(Rroot)--(Ra1);
\node at (10.9,.65) {$d_2$};
\draw[thickedge] (Rl2)--(Ra1)--(Rl1);
\draw[thickedge] (Rl3)--(Ra2)--(Rl4);
\node at (10.4,-.25) {$g_2$};
\node at (10.4,-.75) {$f_2$};
\node at (11,-.4) {$e_2$};
\draw[thickedge] (12.25,0)--(Rroot);
\node[vertexc] at (11.825,0) {};
\node at (12,.15) {$a_2$};
\node at (11.7,.2) {$b_2$};
\draw[thickedge] (Ll1)--(Rl4);
\node at (9.5,.9) {$x$};
\draw[thickedge] (Ll2)--(Rl2);
\node at (9.3,.25) {$m_1$};
\draw[thickedge] (Ll3)--(Rl3);
\node at (9.3,-.25) {$m_2$};
\draw[thickedge] (Ll4)--(Rl1);
\node at (9.5,-.9) {$y$};
\end{tikzpicture}
\end{center}
 \end{proof}

 \begin{theorem}\label{th:k2} Let $\T$ be a tanglegram with $\XX_{\T}=\{X\}$. If $\T[X]\simeq\K_2$, then $\Crt(\T)=1$.
  \end{theorem}

 \begin{proof}
 Since $\XX_{\T}\ne\emptyset$, $\Crt(\T)\ge 1$. To prove the theorem, we need to find a layout of $\T$ with only one crossing.
 
 Let $M$ be the set of edges in $\sigma\setminus X$ that have a scar on one of $d_1$ or $d_2$ in $\T[X]$. 
 Define $N=\sigma\setminus(X\cup M)$.
 By Lemma~\ref{lm:k2possible}, the possible scar-types
 of edges in $N$ in $\T[X]$ are $(a_1,a_2)$, $(a_1,b_2)$, $(b_1,a_2)$, $(a_1,c_2)$, $(c_1,a_2)$ $(b_1,c_2)$ and $(c_1,b_2)$; and
 by Lemmas~\ref{lm:k2possible} and~\ref{lm:k2dscars} either $M=\emptyset$ or
 there is a unique $j\in\{1,2\}$,  such that all edges in $M$ have scars on $d_j$ and the other scar of every edge in $M$ is on $f_{3-j}$ in $\T[X]$.

 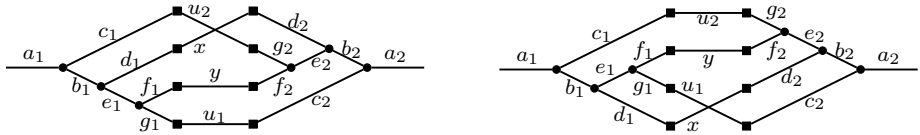
\begin{figure}[http]
\centering 
\begin{tikzpicture}
\node[vertex] (lroot) at (-.5,0) {};
\node[vertex] (la1) at (0,-.25) {};
\node[vertex] (la2) at (.5,-.5) {};
\node[vertexb] (ll1) at (1, .75) {};	
\node[vertexb] (ll2) at (1, 0.25) {};		
\node[vertexb] (ll3) at (1, -0.25) {};		
\node[vertexb] (ll4) at (1, -.75) {};		
\draw[thickedge] (ll1)--(lroot)--(la1);
\node at (.1,.45) {$c_1$};
\node at (-.25,.-.25) {$b_1$};
\draw[thickedge] (ll2)--(la1)--(la2);
\node at (.4,.1) {$d_1$};
\node at (.15,-.5) {$e_1$};
\draw[thickedge] (ll3)--(la2)--(ll4);
\node at (.65,-.75) {$g_1$};
\node at (.65,-.25) {$f_1$};
\draw[thickedge] (-1.25,0)--(lroot);
\node at (-.875,.15) {$a_1$};
\node[vertex] (rroot) at (3.5,0) {};
\node[vertex] (ra1) at (3,.25) {};
\node[vertex] (ra2) at (2.5,0) {};
\node[vertexb] (rl1) at (2, -.75) {};	
\node[vertexb] (rl2) at (2, 0.75) {};		
\node[vertexb] (rl3) at (2, -0.25) {};		
\node[vertexb] (rl4) at (2, .25) {};		
\draw[thickedge] (rl1)--(rroot)--(ra1);
\node at (2.9,-.45) {$c_2$};
\node at (3.3,.25) {$b_2$};
\draw[thickedge] (rl2)--(ra1)--(ra2);
\node at (2.6,.6) {$d_2$};
\node at (2.9,.05) {$e_2$};
\draw[thickedge] (rl3)--(ra2)--(rl4);
\node at (2.4,.25) {$g_2$};
\node at (2.4,-.25) {$f_2$};
\draw[thickedge] (4.25,0)--(rroot);
\node at (3.875,.15) {$a_2$};
\draw[thickedge] (ll1)--(rl4);
\node at (1.3,.75) {$u_2$};
\draw[thickedge] (ll2)--(rl2);
\node at (1.3,.25) {$x$};
\draw[thickedge] (ll3)--(rl3);
\node at (1.5,-.1) {$y$};
\draw[thickedge] (ll4)--(rl1);
\node at (1.5,-.65) {$u_1$};
\end{tikzpicture}
\quad\quad
\begin{tikzpicture}
\node[vertex] (lroot) at (-.5,0) {};
\node[vertex] (la1) at (0,-.25) {};
\node[vertex] (la2) at (.5,0) {};
\node[vertexb] (ll1) at (1, .75) {};	
\node[vertexb] (ll2) at (1, -.75) {};		
\node[vertexb] (ll3) at (1, 0.25) {};		
\node[vertexb] (ll4) at (1, -.25) {};		
\draw[thickedge] (ll1)--(lroot)--(la1);
\node at (.1,.45) {$c_1$};
\node at (-.25,.-.25) {$b_1$};
\draw[thickedge] (ll2)--(la1)--(la2);
\node at (.4,-.6) {$d_1$};
\node at (.15,0) {$e_1$};
\draw[thickedge] (ll3)--(la2)--(ll4);
\node at (.65,-.25) {$g_1$};
\node at (.65,.25) {$f_1$};
\draw[thickedge] (-1.25,0)--(lroot);
\node at (-.875,.15) {$a_1$};
\node[vertex] (rroot) at (3.5,0) {};
\node[vertex] (ra1) at (3,.25) {};
\node[vertex] (ra2) at (2.5,.5) {};
\node[vertexb] (rl1) at (2, -.75) {};	
\node[vertexb] (rl2) at (2, -0.25) {};		
\node[vertexb] (rl3) at (2, 0.25) {};		
\node[vertexb] (rl4) at (2, .75) {};		
\draw[thickedge] (rl1)--(rroot)--(ra1);
\node at (2.9,-.45) {$c_2$};
\node at (3.3,.25) {$b_2$};
\draw[thickedge] (rl2)--(ra1)--(ra2);
\node at (2.6,-.1) {$d_2$};
\node at (2.9,.45) {$e_2$};
\draw[thickedge] (rl3)--(ra2)--(rl4);
\node at (2.4,.75) {$g_2$};
\node at (2.4,.25) {$f_2$};
\draw[thickedge] (4.25,0)--(rroot);
\node at (3.875,.15) {$a_2$};
\draw[thickedge] (ll1)--(rl4);
\node at (1.5,.65) {$u_2$};
\draw[thickedge] (ll2)--(rl2);
\node at (1.5,.1) {$y$};
\draw[thickedge] (ll3)--(rl3);
\node at (1.3,-.25) {$u_1$};
\draw[thickedge] (ll4)--(rl1);
\node at (1.3,-.75) {$x$};
\end{tikzpicture}
\caption{Two more optimal drawings of $\K_2$ with standardized edge-labeling.} 
\label{fig:alternates}
\end{figure} 

For the three cases outlined above, different sublayouts of $\T[X]\simeq\K_2$ will arise as sublayouts. For $M=\emptyset$, the layout on Figure~\ref{fig:labels}, for $j=1$ the layout
on the left side of Figure~\ref{fig:alternates}, and finally for $j=2$ the layout
on the right side of Figure~\ref{fig:alternates} will emerge.

\begin{figure}[http]
\centering 
\begin{tikzpicture}[scale=2]
\fill[gray!50] (3.5,0)--(2,.75)--(1,.75)--(-.5,0)--(1,-.75)--(2,-.75)--(3.5,0);			
\node[vertex] (lroot) at (-.5,0) {};
\node[vertexc] (la1) at (0,-.25) {$p_1$};
\node[vertexc] (la2) at (.5,-.5) {$p_2$};
\node[vertexb] (ll1) at (1, .75) {};	
\node[vertexb] (ll2) at (1, 0.25) {};		
\node[vertexb] (ll3) at (1, -0.25) {};		
\node[vertexb] (ll4) at (1, -.75) {};		
\draw[thickedge] (ll1)--(lroot)--(la1);
\node at (.1,.45) {$c_1$};
\node at (-.25,.-.25) {$b_1$};
\draw[thickedge] (la1)--(la2);
\draw[thickedge, dotted] (ll2)--(la1);
\node at (.4,.1) {$d_1$};
\node at (.15,-.5) {$e_1$};
\draw[thickedge] (ll3)--(la2)--(ll4);
\node at (.65,-.75) {$g_1$};
\node at (.65,-.25) {$f_1$};
\draw[thickedge] (-1.25,0)--(lroot);
\node at (-.875,.15) {$a_1$};
\node[vertex] (rroot) at (3.5,0) {};
\node[vertexc] (ra1) at (3,.25) {$q_1$};
\node[vertexc] (ra2) at (2.5,.5) {$q_2$};
\node[vertexb] (rl1) at (2, -.75) {};	
\node[vertexb] (rl2) at (2, -0.25) {};		
\node[vertexb] (rl3) at (2, 0.25) {};		
\node[vertexb] (rl4) at (2, .75) {};		
\draw[thickedge] (rl1)--(rroot)--(ra1);
\node at (2.9,-.45) {$c_2$};
\node at (3.3,.25) {$b_2$};
\draw[thickedge] (ra1)--(ra2);
\draw[thickedge, dotted] (rl2)--(ra1);
\node at (2.6,-.1) {$d_2$};
\node at (2.9,.45) {$e_2$};
\draw[thickedge] (rl3)--(ra2)--(rl4);
\node at (2.4,.75) {$g_2$};
\node at (2.4,.25) {$f_2$};
\draw[thickedge] (4.25,0)--(rroot);
\node at (3.875,.15) {$a_2$};
\draw[thickedge] (ll1)--(rl4);
\node at (1.5,.9) {$u_2$};
\draw[thickedge, dotted] (ll2)--(rl2);
\node at (1.3,.25) {$x$};
\draw[thickedge] (ll3)--(rl3);
\node at (1.3,-.25) {$y$};
\draw[thickedge] (ll4)--(rl1);
\node at (1.5,-.9) {$u_1$};
\end{tikzpicture}
\caption{Induced subdrawing of $\T[\{u_1,u_2,y\}]$ in  $\D^{\prime}$.} 
\label{fig:k2sub}
\end{figure}
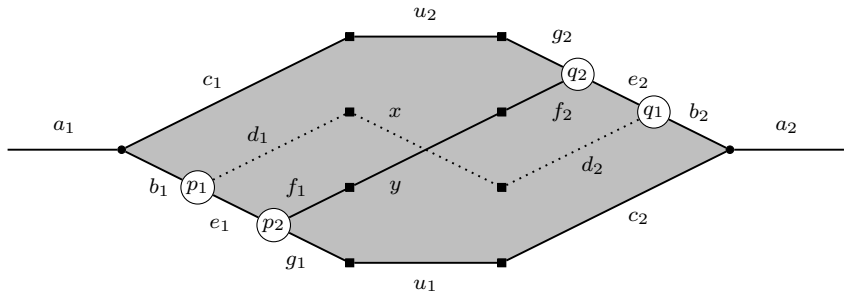

For any edge $\beta\in\sigma$, we will use $\lambda_{L,\beta}$ and $\lambda_{R,\beta}$ to denote the  leaves incident to this edge in the left- and right-tree, respectively.  Without loss of generality we can assume that $X=\{u_1,u_2,x,y\}$ in  the standardized edge-labeling of $\K_2\simeq\T[X]$ given in  Figure~\ref{fig:labels}. We set
$p_1=\lambda_{L,u_1}\land_{L_{\T}}\lambda_{L,x}$, $p_2=\lambda_{L,u_1}\land_{L_{\T}}\lambda_{L,y}$, 
$q_1=\lambda_{R,u_2}\land_{R_{\T}}\lambda_{R,x}$, $q_2=\lambda_{R,u_2}\land_{R_{\T}}\lambda_{R,y}$ (see Figure~\ref{fig:k2sub}). 
By the second part of Lemma~\ref{lm:k2possible}, for $i\in\{1,2\}$ we have  $P_{g_i}=g_i$ and $P_{e_i}=e_i$ for the representative paths as in
Definition~\ref{rep_path}.

Consider the tanglegram  $\T^{\prime}=\T[\sigma_{\T}\setminus(M\cup\{x\})]$ and  
note that $\sigma_{\T^{\prime}}=N\cup\{u_1,u_2,y\}$.
As $\XX_{\T^{\prime}}=\emptyset$, Theorem~\ref{th:tkurat} gives $\Crt(\T^{\prime})=0$. 
Consider a planar layout $\D^{\prime}$ of $\T^{\prime}$.  
$\D^{\prime}$ (or its mirror image to the $x$ axis)
induces a planar sublayout, in which 
$u_1$ lies below $u_2$, and consequently the induced sublayout of $\T[\{u_1,u_2,y\}]$ is as shown on Figure~\ref{fig:k2sub}.
From now on, this layout is $\D^{\prime}$. 
Let the representation of  the layout $\D^{\prime}$ be denoted by
\begin{equation} \label{primeorder}
(\vec{\ell}_{L,\D^{\prime}},\vec{\ell}_{R,D^{\prime}}). 
\end{equation}

As $\D^{\prime}$ is planar, none of the edges in $N$  enter the interior of the shaded region on Figure~\ref{fig:k2sub}. Consequently,
$\lambda_{L,u_1},\lambda_{L,y},\lambda_{L,u_2}$ and $\lambda_{R,u_2},\lambda_{R,y},\lambda_{R,u_1}$ are contiguous 
subsequences of $\vec{\ell}_{L,\D^{\prime}}$ and $\vec{\ell}_{R,D^{\prime}}$, and for any $\alpha,\beta\in\sigma_{\T^{\prime}}$ we have 
$\lambda_{L,\alpha}\leq_{\vec{\ell}_{L,\D^{\prime}}} \lambda_{L,\beta}$ precisely when
$\lambda_{R,\beta}\leq_{\vec{\ell}_{R,\D^{\prime}}} \lambda_{R,\alpha}$. Define the (possibly empty) sequences 
$\vec{k}_1,\vec{k}_2,\vec{k}_3,\vec{k}_4$  by 
\begin{equation} \label{eq:vektor}
\vec{\ell}_{L,\D^{\prime}}=(\vec{k}_1,\lambda_{L,u_1},\lambda_{L,y},\lambda_{L,u_2},\vec{k}_2),\,\,\,
\hbox{\ and \ }
 \vec{\ell}_{R,\D^{\prime}}=(\vec{k}_3,\lambda_{R,u_2},\lambda_{R,y},\lambda_{R,u_1},\vec{k}_4).
\end{equation}
Consider the case $M=\emptyset$, in other words when $\T^{\prime}=\T[\sigma\setminus\{x\}]$. By the last part of Lemma~\ref{lm:k2possible}, we have that for $i\in\{1,2\}$ that $P_{d_i}=d_i$ and
$P_{f_i}=f_{i}$.
We set the leaf-orders
$\vec{\ell}_{L}=(\vec{k}_1,\lambda_{L,u_1},\lambda_{L,y},\lambda_{L,x},\lambda_{L,u_2},\vec{k}_2)$ and 
$\vec{\ell}_{R}=(\vec{k}_3,\lambda_{R,u_2},\lambda_{R,y}\lambda_{R,x},\lambda_{R,u_1},\vec{k}_4)$, i.e., 
we insert $\lambda_{L,x}$ between $\lambda_{L,y}$ and $\lambda_{L,u_2}$ in $\vec{\ell}_{L,\D^{\prime}}$ and we 
insert $\lambda_{R,x}$ between $\lambda_{R,y}$ and $\lambda_{R,u_1}$
in $\vec{\ell}_{R,D^{\prime}}$ (See the dotted edges in Figure~\ref{fig:k2sub}). Using Lemma~\ref{lm:consistent}, it is easy to see that
$\vec{\ell}_L$ and $\vec{\ell}_R$ are consistent leaf-orders of $L_{\T}$ and $R_{\T}$; we will show
the details for $\vec{\ell}_L$ here. The descendant of $p_2$ are $\lambda_{L,u_1}$ and $\lambda_{L,y}$ that make a contiguous
subsequence of $\vec{\ell}_{L}$.
Let $v\in\II(L_{\T})\setminus \{p_2\}$. As $f_1,g_1,d_1,e_1$ are not subdivided in $\T$,
$\lambda_{L,x}$ is a leaf-descendant of $v$ precisely when $\lambda_{L,u_1},\lambda_{L,y}$ are leaf-descendants of $v$.
Since the leaf-descendants of $v$ in $\T^{\prime}$ form a contiguous subsequence of $\vec{\ell}_{L,\D^{\prime}}$, $\vec{\ell}$ is consistent.
Let $\D$ be the layout of $\T$ with representation $(\vec{\ell}_L,\vec{\ell}_R)$.
As in the proof of Theorem~\ref{th:k1}, the formula (\ref{eq:crosslayout}) gives $\Cr(\D)\leq 1$.

\begin{figure}[http]
\centering 
\begin{tikzpicture}[scale=1.5]
\fill[gray!50] (3,0.25)--(2,.75)--(1,.75)--(-.5,0)--(.5,-.5)--(1,-.25)--(2,-.25)--(3,0.25);			
\node[vertex] (lroot) at (-1,0) {};
\node[vertexc] (la1) at (-.5,0) {$p_1$};
\node[vertexc] (la2) at (.5,-.5) {$p_2$};
\node[vertexb] (ll1) at (1, 1) {};	
\node[vertexb] (ll2) at (1, 0.75) {};		
\node[vertexb] (ll3) at (1, -0.25) {};		
\node[vertexb] (ll4) at (1, -.75) {};		
\draw[thickedge,dotted] (ll1)--(lroot);
\draw[thickedge] (lroot) --(la1);
\node at (-.1,.55) {$c_1$};
\node at (-.75,.-.15) {$b_1$};
\draw[thickedge] (la1)--(la2);
\draw[thickedge] (ll2)--(la1);
\node at (.4,.3) {$d_1$};
\node at (0,-.35) {$e_1$};
\draw[thickedge] (ll3)--(la2)--(ll4);
\node at (.65,-.75) {$g_1$};
\node at (.65,-.25) {$f_1$};
\draw[thickedge] (-1.5,0)--(lroot);
\node at (-1.25,.15) {$a_1$};
\node[vertex] (rroot) at (3.5,0) {};
\node[vertexc] (ra1) at (3,.25) {$q_1$};
\node[vertexc] (ra2) at (2.5,0) {$q_2$};
\node[vertexb] (rl1) at (2, -.75) {};	
\node[vertexb] (rl2) at (2, 0.75) {};		
\node[vertexb] (rl3) at (2, -0.25) {};		
\node[vertexb] (rl4) at (2, .25) {};		
\draw[thickedge] (rl1)--(rroot)--(ra1);
\node at (2.9,-.45) {$c_2$};
\node at (3.3,.25) {$b_2$};
\draw[thickedge] (ra1)--(ra2);
\draw[thickedge] (rl2)--(ra1);
\node at (2.45,.65) {$d_2$};
\node at (2.8,0) {$e_2$};
\draw[thickedge] (rl3)--(ra2);
\draw[thickedge,dotted] (ra2)--(rl4);
\node at (2.3,.25) {$g_2$};
\node at (2.3,-.25) {$f_2$};
\draw[thickedge] (4.25,0)--(rroot);
\node at (3.875,.15) {$a_2$};
\draw[thickedge,dotted] (ll1)--(rl4);
\node at (1.6,.85) {$x$};
\draw[thickedge] (ll2)--(rl2);
\node at (1.5,-.15) {$y$};
\draw[thickedge] (ll3)--(rl3);
\node at (1.65,.35) {$u_2$};
\draw[thickedge] (ll4)--(rl1);
\node at (1.5,-.9) {$u_1$};
\end{tikzpicture}
\caption{Schematic drawing of  $\D^{\star}$.} 
\label{fig:k2star}
\end{figure}
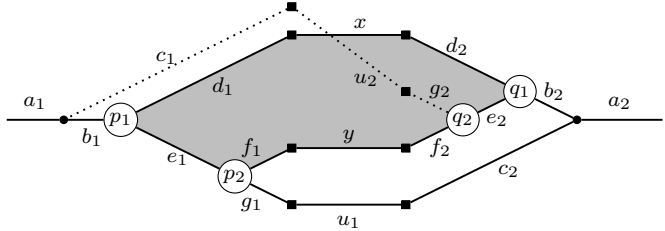

Consider now the case 
$M\ne\emptyset$.  By Lemma~\ref{lm:k2dscars}, there are two possibilities for the scar-type of the elements of $M$. We  may assume without loss of generality that the edges of $M$ have scar-type $(d_1,f_2)$, as the other case, when the edges of $M$ have a scar-type $(f_1,d_2)$ 
can be handled similarly.  
By Lemma~\ref{lm:k2dscars}, we have $P_{d_2}=d_2$ and $P_{f_1}=f_1$, and hence $d_2,f_1$ are edges of $\T$. By the second part of
Lemma~\ref{lm:k2possible},   $e_1,e_2,g_1,g_2$ are also edges of $\T$.
Define  $\T^{\star}=T[\{x,y,u_1\}\cup M]$. As $u_2\notin\sigma_{\T^{\star}}$, $\XX_{\T^{\star}}=\emptyset$ and consequently $\Crt(T^{\star})=0$.
Let $D^{\star}$ be a planar layout of $T^{\star}$ where $u_1$ lies below $x$. Then the induced sublayout of  
$\T^{\star}[\{u_1,x,y\}]$ in $D^{\star}$  must be as on Figure~\ref{fig:k2star}. Moreover,  for each $m\in M$ the path going through $m$ between
the scars of $m$ in $\T^{\star}$ must lie in the region shaded gray in Figure~\ref{fig:k2star},  with boundary  formed by $e_1f_1yP_{f_2}e_2d_2xP_{d_1}$ in $D^{\star}$. Therefore for the representation $(\vec{\ell}_{L,\D^{\star}},\vec{\ell}_{R,\D^{\star}})$ of $D^{\star}$ we have that
\begin{equation}
\label{eq:vektor2}
\vec{\ell}_{L,D^{\star}}=(\lambda_{L,u_1},\lambda_{L,y},\vec{k}_5,\lambda_{L,x})\,\,\,\hbox{ and }
\,\,\, \vec{\ell}_{R,D^{\star}}=(\lambda_{R,x},\vec{k}_6,\lambda_{R,y},\lambda_{R,u_1}),
\end{equation} 
where $((\vec{k}_5,\lambda_{L,x}),(\lambda_{R,x},\vec{k}_6))$ is the sublayout of
$\overline{\T}=\T^{\star}[M\cup\{x\}]$ induced by $\D^{\star}$. 

Using the $\vec{k}_i$ defined in (\ref{eq:vektor}) and (\ref{eq:vektor2}), set  the leaf-order
$\vec{\ell}_{L}=(\vec{k}_1,\lambda_{L,u_1},\lambda_{L,y},\vec{k}_5,\lambda_{L,x},\lambda_{L,u_2}\vec{k}_2)$. In other words we insert the sequence $(\vec{k}_5,\lambda_{L,x})$ between $\lambda_{L,y}$ and $\lambda_{L,u_2}$ in
$\vec{\ell}_{L,\D^{\prime}}$.
Clearly,
$\vec{\ell}_{L}$ is a total order of $\LL(L_{\T})$.

Using  Lemma~\ref{lm:consistent}, it suffices to show that the  leaf descendants of any $v\in\II(L_{\T})$ make a contiguous subsequence of $\vec{\ell}_L$. Note that $\LL(L_{\T})$ is the disjoint union of $\LL(L_{\T^{\prime}})$ and $\LL(L_{\overline{\T}})$.
We have the following cases to consider:
\begin{itemize}
\item $v\preceq_{L_{\T}} p_1$:
The leaf-descendants of $v$ in $L_{\T^{\prime}}$ form a contiguous subsequence in $\vec{\ell}_{L,\D^{\prime}}$ and
include   $\lambda_{L,u_1}$ and $\lambda_{L,y}$.  
The leaf descendants in $L_{\overline{\T}}$ are the entire leaf set which is listed in the subsequence $(\vec{\ell}_{\overline{D}},\lambda_{L,x})$. Therefore the leaf-descendants of
$v$ form a contiguous subsequence of $\vec{\ell}_L$.
\item $v=p_2$: As $g_1$ and $f_1$ are edges of $\T$, the leaf-descendants of $p_2$ are $\lambda_{L,u_1}$ and $\lambda_{L,y}$, which form a contiguous subsequence in $\vec{\ell}_L$.
\item $p_1\preceq_{L_{\T}} v$ and $v\notin\{p_1,p_2\}$: In this case $v\in\II(L_{\overline{\T}})$, and consequently the leaf-descendants of $v$ form a contiguous subsequence of 
$(\vec{k}_5,\lambda_{L,x})$.
\item $v$ is incomparable to $p_1$ in $\preceq_{L_{\T}}$:  In this case 
$v\notin\II(L_{\overline{\T}})$, 
the leaf-descendants of $v$ are all in $L_{\T^{\prime}}$ and do not include $\lambda_{L,u_1},\lambda_{L,y}$,
therefore they are a contiguous subsequence of either $\vec{k}_1$ or $\vec{k_2}$. 
\end{itemize}
By Lemma~\ref{lm:consistent}, $\vec{\ell}_L$ is a consistent order of
the leaves of $L_{\T}$.

 Using the $\vec{k}_i$ in (\ref{eq:vektor}) and (\ref{eq:vektor2}) again, set
$\vec{\ell}_{R}=(\vec{k}_3,\lambda_{R,x},\lambda_{R,u_2},\vec{k}_6,\lambda_{R,y},\lambda_{R,u_1},\vec{k}_4)$, in other words, insert $\lambda_{R,u_2}$ between
$\lambda_{R,x}$ and $\vec{k}_6$ in $\vec{\ell}_{R,\D^{\prime}}$.
Clearly, $\vec{\ell}_{R}$ is a total order of $\LL(R_{\T})$.

Using  Lemma~\ref{lm:consistent}, it suffices to show that the  leaf descendants of any $v\in\II(R_{\T})$ make a contiguous subsequence 
of $\vec{\ell}_R$. Note that $\LL(R_{\T})$ is the disjoint union of $\LL(R_{\T^{\prime}})$ and $\LL(R_{\overline{\T}})$.
\begin{itemize}
\item $v\preceq_{R_{\T}} q_1$: The leaf-descendants of $v$ in $R_{\T^{\prime}}$ form a contiguous subsequence of $\vec{\ell}_{R,\D^{\prime}}$ and 
include $\lambda_{R,u_2},\lambda_{R,y}$. The set of leaf-descendants of $v$ in $R_{\overline{T}}$ is $\LL(R_{\overline{T}})$.
Therefore the leaf-descendants of $v$ in $\T$ form a contiguous subsequence of $\vec{\ell}_{R}$.
\item $v=q_2$: The set of leaf-descendants of $v$ in $R_{\T}$  is $\{\lambda_{R,u_2},\lambda_{R,y}\}\cup(\LL(T_{\overline{T}})\setminus\{\lambda_{R,x})\}$.
\item $q_2 \preceq_{R_{\T}} v$ and $v\ne q_2$: 
The leaf-descendants of $v$ in $R_{\T}$ are all in $R_{\overline{T}}$ and do not include $\lambda_{R,x}$, therefore they form a contiguous subsequence of
$\vec{k}_6$.
\item $v$ is incomparable to $q_1$ in $\preceq_{R_{\T}}$:  In this case 
$v\notin\II(R_{\overline{\T}})$, 
the leaf-descendants of $v$ are all in $R_{\T^{\prime}}$ and do not include $\lambda_{R,u_2},\lambda_{R,y}$,
therefore they are a contiguous subsequence of either $\vec{k}_3$ or $\vec{k_4}$. 
\end{itemize}

Therefore $(\vec{\ell}_L,\vec{\ell}_R)$ is a representation of a layout $\D$ of $\T$. Formula (\ref{eq:crosslayout}) and the properties of the layout 
$\D^{\prime}$ and $\D^{\star}$ give that $\Cr(\D)\leq 1$.
\end{proof}

\section*{Acknowledgments}
 Part of this work was done while EC and LS were in residence at the Institute for Computational and Experimental Research in Mathematics (ICERM) in Providence (RI, USA) during the Theory, Methods, and Applications of Quantitative Phylogenomics program (supported by grant DMS-1929284 of the National Science Foundation (NSF)).


\end{document}